\documentclass[11pt,twoside]{article}
\usepackage{amsmath,amssymb,epsfig,subfigure,graphicx,epstopdf,blkarray,color}
\usepackage{multirow,threeparttable}
\usepackage{algorithm}
\usepackage{algorithmic}
\usepackage{enumerate}
\usepackage{rotating,booktabs}
\usepackage[colorlinks=true]{hyperref}
\hypersetup{urlcolor=blue,citecolor=blue,linkcolor=blue}

\newtheorem{proposition}{Proposition}[section]
\newtheorem{theorem}[proposition]{Theorem}
\newtheorem{lemma}[proposition]{Lemma}

\newtheorem{definition}[proposition]{Definition}
\newtheorem{remark}[proposition]{Remark}
\newtheorem{example}[proposition]{Example}

\newenvironment{proof}{{\noindent \em Proof.}}{\hfill $\fbox{}$ \vspace*{5mm}}
\numberwithin{equation}{section}
\numberwithin{algorithm}{section}
\numberwithin{table}{section}
\numberwithin{figure}{section}

\newcommand{\bb}{{\bf b}}
\newcommand{\bc}{{\bf c}}
\newcommand{\bff}{{\bf f}}

\newcommand{\bm}{{\bf m}}

\newcommand{\bu}{{\bf u}}

\newcommand{\bx}{{\bf x}}
\newcommand{\by}{{\bf y}}

\newcommand{\ca}{{\cal A}}

\newcommand{\E}{{\mathbb{E}}}
\newcommand{\BP}{{\mathbb{P}}}
\newcommand{\Cmn}{{\mathbb{C}^{m\times n}}}

\newcommand{\R}{{\mathbb{R}}}

\newcommand{\Rn}{{\mathbb{R}^{n}}}
\newcommand{\Rmm}{{\mathbb{R}^{m\times m}}}
\newcommand{\Rmn}{{\mathbb{R}^{m\times n}}}
\newcommand{\Rnn}{{\mathbb{R}^{n\times n}}}
\newcommand{\bd}{\boldsymbol}
\newcommand{\diag}{{\rm diag}}
\newcommand{\rank}{{\rm rank}}

\newcommand{\BE}{\begin{equation}}
\newcommand{\EE}{\end{equation}}

\begin{document}
\title{Two-sided uniformly randomized GSVD for large-scale discrete ill-posed problems with Tikhonov regularizations}
\author{Weiwei Xu\thanks{School of Mathematics and Statistics, Nanjing University of Information Science and Technology, Nanjing 210044, the Peng Cheng Laboratory, Shenzhen 518055, and the Pazhou Laboratory (Huangpu), Guangzhou 510555, People's Republic of China (wwxu@nuist.edu.cn). }
\and Weijie Shen\thanks{School of Mathematics and Statistics, Nanjing University of Information Science and Technology, Nanjing 210044, People's Republic of China (swj@nuist.edu.cn). }
\and Zheng-Jian Bai\thanks{Corresponding author. School of Mathematical Sciences,  Xiamen University, Xiamen 361005, People's Republic of China (zjbai@xmu.edu.cn). The research of this author was partially supported by the National Natural Science Foundation of China grant 12371382.}}

\maketitle
\begin{abstract}
The generalized singular value decomposition (GSVD) is a powerful tool for solving discrete ill-posed problems. In this paper, we propose a  two-sided uniformly randomized GSVD algorithm for solving the large-scale discrete ill-posed problem with  the general Tikhonov regularization. Based on  two-sided uniform random sampling, the proposed  algorithm can improve the efficiency with less computing time and  memory requirement and obtain expected accuracy. The error analysis for the proposed  algorithm is also derived. Finally, we report some numerical examples to illustrate the efficiency of the proposed algorithm.
\end{abstract}
\section{Introduction}
We consider the numerical solution of a discrete ill-posed problem $A\bx=\bb$, where $A\in \mathbb{R}^{m\times n}$ is large and very ill-conditioned in the sense that its singular values decay gradually to zero and $\bb\in \mathbb{R}^{m}$. Such problem arises in many fields including  mathematical physics \cite{TA77}, engineering \cite{BP88,N86}, and  astronomy \cite{CB86}, etc.

There exists a large of literature on regularization methods for  discrete ill-posed problems (see for instance \cite{EH96,H10,H86}). In particular,
Tikhonov regularization method is a popular and effective method for solving discrete ill-posed problems. The general Tikhonov regularization takes the form of \cite{tikhonov}
\begin{equation} \label{tikhonov}
\min_{\bx\in \Rn} \|A\bx-\bb\|^2+\lambda^2\|L\bx\|^2,
\end{equation}
where $\lambda\in\R$ is a regularization parameter. Here,  $L\in\R^{p\times n}$ is generally a discrete approximation to some derivative operator.
When $\mathcal{N}(A) \bigcap \mathcal{N}(L)=\{{\bf 0}\}$, there exists a unique solution to \eqref{tikhonov} for all $\lambda>0$, which is given by
\BE\label{tsol}
\bx_\lambda=
([A^T,\lambda L^T]^T)^\dag [\bb^T,{\bf 0}^T]^T
=(A^TA+\lambda^2L^TL)^{-1}A^T\bb.
\EE

For small and medium scale discrete ill-posed problems, one may employ the truncated singular value decomposition
(TSVD) \cite{H87}, the generalized SVD (GSVD) and truncated GSVD (TGSVD) \cite{tikhonov}, and a scaling GSVD \cite{rescaleGSVD}. However, for large-scale ill-posed problems, the GSVD is not applicable due to its high computational complexity and storage requirement.

During the last decade, randomized algorithms have been developed for solving large-scale problem \eqref{tikhonov}. Xiang and Zou \cite{XZ13} presented a randomized SVD algorithm by employing some basic randomized algorithms from \cite{random}. Xiang and Zou \cite{xiang2015} further gave a randomized GSVD (RGSVD)  algorithm.  Wei et al. \cite{2016} provided a new RGSVD algorithm for the cases of $m\ge n$ and $m<n$. Jia et al.~\cite{jia2018} proposed a modified truncated RSVD (MTRSVD) algorithm.  Although these randomized algorithms are faster than classical regularization methods, they only reduce the number of rows or columns and the reduced problem is still large-scale. 

In this paper, we propose a  two-sided uniformly randomized GSVD algorithm for solving the large-scale discrete ill-posed problem with  the general Tikhonov regularization. We first present some preliminary results  on uniform random matrices. Then we introduce a blockwise adaptive randomized range finder for $A$ based on a uniform sampling technique. The novelty lies in the application of uniform random matrices for random sampling and an approximate basis for the range of $A$ can be obtained within the prescribed tolerance. Also,  we present a new GSVD algorithm for a  full-rank Grassman matrix pair (GMP).  Next, a two-sided uniformly randomized GSVD algorithm is provided for solving the  large-scale Tikhonov regularization problem \eqref{tikhonov}. Our  algorithm can reduce the computing time and the memory storage and achieve expected accuracy. The error analysis of the proposed algorithm is established. Finally,  some numerical examples are presented to illustrate the efficiency of the  proposed algorithm. 

The rest of this paper is organized as follows. In Section \ref{sec2}, we give some  preliminary results  on uniform random matrices. In Section \ref{sec3}, a randomized scheme for  the range approximation via uniform sampling is provided. In Section \ref{sec4}  we  present  a new GSVD algorithm for a full-rank  GMP.  In Section \ref{sec5}, we propose a  two-sided uniformly randomized GSVD algorithm  for the  large-scale Tikhonov regularization problem \eqref{tikhonov} and the corresponding  error analysis is also discussed. In Section \ref{sec6},  we give  some numerical examples to illustrate the efficiency of the proposed algorithm. Finally, we give some concluding remarks in Section \ref{sec7}.

Throughout this paper, we use the following notations. Let $\Cmn$ be the set of all $m\times n$ complex matrices, Let $\R$ and $\Rmn$ be the sets of all real numbers and all  $m\times n$ real matrices, respectively, where $\Rn=\R^{n\times 1}$. Let $\Rn$ be equipped with the  Euclidean inner product $\langle \cdot,\cdot\rangle$ with the induced norm  $\|\cdot\|$.  Denote by $\mathcal{R}(\cdot)$ and  $\mathcal{N}(\cdot)$  the range space and null space of a matrix. Let $\|\cdot\|_2$ and $\|\cdot\|_F$ be the spectral norm and Frobenius norm of a matrix, respectively.   Let $\lfloor\cdot\rfloor$ be the largest integer less than or equal to a real number. Let $I_n$ be the $n\times n$ identity matrix. Let $B^T$, $B^{-1}$, $B^\dag$, $(B)_{ij}$ (or $b_{ij}$), $B(i,:)$, $B(:,j)$, $B(i:j,:)$, $B(:,i:j)$, and $B(i_1:i_2,j_1:j_2)$  denote the transpose, the inverse, and the Moore--Penrose inverse, the $(i,j)$ entry, the $i$-th row, the $j$-th column, a matrix composed of rows $i$ to $j$, a matrix composed of columns $i$ to $j$, and  a submatrix comprising the intersection of rows $i_1$ to $i_2$ and columns $j_1$ to $j_2$, of a matrix $B$, respectively.  Let $\sigma_1(B)\ge\sigma_2(B)\ge\cdots\ge\sigma_{\min\{m,n\}}(B)\ge 0$ be the singular values of a matrix $B\in\Rmn$. A matrix $B\in\Rmn$ is called a uniform random matrix if all its entries are independent  uniform random variables  over an interval. Let $\BP\{\cdot\}$ and  $\E(\cdot)$ denote the probability of an event and the expectation of a random variable, respectively. 
\section{Preliminaries}\label{sec2}
In this section we present some preliminary results  on uniform random matrices.
We first derive the following result on a linear transformation of a random matrix with some special properties.
\begin{lemma}\label{t1}
Let $B\in \R^{m\times \ell}$ and $C\in \R^{m\times n}$ be two constant matrices with $C$ being of full row rank.  Let $\Omega\in\R^{n\times \ell}$  {\rm($n\ge \ell$)} be a random matrix whose entries are all independent random variables such that, for any $c\in\R$,
\begin{equation}\label{e23100101}
\BP\{\omega_{ij}=c\}=0, \quad i=1,\ldots, n,\quad  j=1,\ldots, \ell.
\end{equation}
Then for any $s\times t$ submatrix $\Phi$ of the random matrix $M=B+C \Omega$, one has
$\rank(\Phi)=\min{\left\{s,t\right\}}$ with probability one, where $1\leq s \leq m$  and $1\leq t\leq \ell$.
\end{lemma}
\begin{proof}
 We only need to prove that for any $s\times s$ submatrix $\Phi$ of the random matrix $M$, we have
\begin{equation}\label{e23100203}
\BP\{\det(\Phi)\neq 0\}=1
\end{equation}
with probability one,
where $1\leq s\leq \min{\left\{m,\ell\right\}}$.

We show \eqref{e23100203} by the mathematical induction. We first show that \eqref{e23100203} holds for $s=1$. Let $C=[\bc_1,\ldots,\bc_m]^T$ and $\Omega=[\bd\omega_1,\ldots, \bd\omega_\ell]$. Then we have $m_{ij}=b_{ij}+\bc_i^T \bd\omega_j$ for all $1\le i\le m$ and $1\le j\le \ell$. Let $1\le i\le m$ and $1\le j\le \ell$ be fixed. By hypothesis, $\rank(C)=m$, which indicates $\bc_i\neq {\bf 0}$. Without loss of generality, we assume that $c_{i1}\neq 0$. Thus,
\[
\{m_{ij}=0\}=\Big\{\omega_{1j}=-c_{i1}^{-1}\Big(b_{ij}+\sum\limits_{l=2}^n c_{il} \omega_{lj}\Big)\Big \}.
\]
By hypothesis, $\omega_{lj}$ are independent random variables  for all $1\le l\le n$. This, together with  (\ref{e23100101}), implies that $\BP\{m_{ij}=0\}=0$. This means that \eqref{e23100203} holds for $s=1$ with probability one.

Next, assume that  \eqref{e23100203} holds for $s=k-1$. We show that  \eqref{e23100203} holds for  $s=k$. Without loss of generality, suppose that $\Phi$  is the submatrix of $M$ comprising  the elements at the intersection of rows $1$ to $k$ and columns $1$ to $k$. By the Laplace expansion, we have
\begin{equation}\label{e23100204}
\det(\Phi)=\sum\limits_{i=1}^{k}(-1)^{i+1}\phi_{i1}\det\left(\Phi_{i1}\right)=\sum\limits_{i=1}^{k}(-1)^{i+1}m_{i1}\det\left(\Phi_{i1}\right),
\end{equation}
where $\Phi_{i1}$ is the $(k-1)\times (k-1)$ submatrix of $\Phi$ obtained by deleting row $i$ and column $1$. Let $C_k$ be the $k\times n$ submatrix of $C$ composed of rows $1$ to $k$.  Then $\rank(C_k)=k$ since $C$ is of full row rank. Without loss of generality, we assume that
$C_k=[C_{k1}, C_{k2}]$, where $C_{k1}\in\R^{k\times k}$ is nonsingular.
Also, let
\[
\bd\rho=\left(\begin{matrix}(-1)^{1+1}\det(\Phi_{11})\\ \vdots\\ (-1)^{k+1}\det(\Phi_{k1})\\\end{matrix}\right), \quad \tilde{\bm}_k=\left(\begin{matrix}m_{11}\\ \vdots\\ m_{k1}\\\end{matrix}\right), \quad
\tilde{\bb}_k=\left(\begin{matrix}b_{11}\\ \vdots\\ b_{k1}\\\end{matrix}\right),
\]
and $\bd\omega_1=[(\bd\omega_1^{(1)} )^T,  (\bd\omega_1^{(2)})^T]^T$ with $\bd\omega_1^{(1)}\in\R^k$.
We note that $\tilde{\bm}_k=\tilde{\bb}_k+C_k\bd\omega_1=\tilde{\bb}_k+C_{k1}\bd\omega_1^{(1)} +C_{k2}\bd\omega_1^{(2)}$. From (\ref{e23100204}) we have
\[
\det(\Phi)=\bd\rho^T \tilde{\bm}_k=\bd\rho^T (\tilde{\bb}_k+C_{k2}\bd\omega_{1}^{(2)})+\bd\rho^T C_{k1}\bd\omega_1^{(1)}.
\]

By inductive hypothesis, we know that the determinants $\{\det(\Phi_{i1})\}_{i=1}^k$ are all nonzero with probability one and thus  $\bd\rho\neq{\bf 0}$ with the probability one.  Then  $\bd\sigma^T=\bd\rho^T C_{k1}\neq{\bf 0}^T$ since $C_{k1}$ is  nonsingular. Without loss of generality, assume that  $\bd\sigma=\left(\sigma_1,\ldots,\sigma_k\right)^T$, where $\sigma_1\neq 0$. Thus,
\[
\{\bd\rho^T \tilde{\bm}_k=0\}=\Big\{\omega_{11}=-\sigma_1^{-1}\Big(\bd\rho^T (\tilde{\bb}_k+C_{k2}\bd\omega_{1}^{(2)})+\sum\limits_{i=2}^k \sigma_i \omega_{i1}\Big)\Big\}.
\]
By hypothesis, $\omega_{i1}$ are independent random variables  for all $1\le i\le n$. This, together with  (\ref{e23100101}), implies that $\BP\{\det(\Phi)=0\}=\BP\{\bd\rho^T \tilde{\bm}_k=0\}=0$. This means that \eqref{e23100203} holds for $s=k$ with probability one. \end{proof}
\begin{remark} We observe from Lemma  {\rm\ref{t1}} that the condition \eqref{e23100101} is satisfied if $\Omega$ is a  random matrix whose entries   are all independent and continuous random variables {\rm(}e.g.,  a Gaussian/uniform random matrix{\rm)}.
\end{remark}

The following lemma gives the relationship between the rank of a matrix and that of its a sampling matrix drawn by a uniform random test matrix.
\begin{lemma}\label{t2}
Let $B\in \Rmn$ and $C\in \R^{n\times \ell}$ be two constant matrices with $n\ge \ell$. Suppose that $\Omega\in\mathbb{R}^{n\times \ell}$ is a uniform random matrix with $\omega_{ij}\in[-a,a]$ for some $a>0$.
Then
$
\rank(B(C+\Omega))= \min\{\rank(B), \ell\}
$
with probability one.
\end{lemma}
\begin{proof}
Let $\rank(B)=r\le\min\{m,n\}$. Then we can reduce $B$ to the following normal form:
\[
B = \widehat{P}\left[\begin{matrix}I_r&0\\0&0\\\end{matrix}\right]\widehat{Q},\quad\mbox{$\widehat{P}\in\Rmm$ and $\widehat{Q}\in\Rnn$ are two nonsingular  matrices}.
\]
Denote $\widehat{Q}=[\widehat{Q}_{1}^T, \widehat{Q}_{2}^T]^T$, where $\widehat{Q}_{1}\in \R^{r\times n}$ with $\rank(\widehat{Q}_1)=r$. Then we have $\rank(B(C+\Omega))=\rank (\widehat{Q}_1(C+\Omega))=\rank (\widehat{Q}_1C+\widehat{Q}_1\Omega)$.  Since $\widehat{Q}_1$ is of full row rank, using Lemma \ref{t1}, we know that $\rank(B(C+\Omega))=\rank (\widehat{Q}_1C+\widehat{Q}_1\Omega)=\min\{r,\ell\}$ with probability one.
\end{proof}

On the rotation of a uniform random matrix, we have the following property.
\begin{lemma}\label{unif:rotation}
Let $C\in\R^{n\times k}$ be a constant matrix, which has orthonormal columns $(k\le n)$. Suppose that $\Omega\in\mathbb{R}^{n\times \ell}$ is a uniform random matrix with $\omega_{ij}\in[-a,a]$ for some $a>0$. Let $H=C^T\Omega$. Then $\E (h_{ij})=0$ and $\E (h_{ij}^2) = a^2/3$ for all $1\leq i \leq k$ and $1\leq j \leq \ell$.
\end{lemma}
\begin{proof}
Since $\Omega$ is an $n\times \ell$ uniform random matrix with $\omega_{ij}\in[-a,a]$ for some $a>0$, we know that $\mathbb{E}(\omega_{ij})=0$, $\mathbb{E}(\omega_{ij}^2)=a^2/3$ for all $1\leq i \leq n$  and $1\leq j\leq \ell$, and  $\mathbb{E}(\omega_{ij}\omega_{st})=\E(\omega_{ij})\E(\omega_{st})=0$ for all $1\leq i,s \leq n$  and $1\leq j,t\leq \ell$ with $(i,j)\neq(s,t)$.

From $H=C^T\Omega$ we have $\E (H) = C^T\E(\Omega)=0$. In addition, for any $1\le i\le k$ and $1\le j\le \ell$, we have
\[
\E (h_{ij}^2) = \E\Big(\Big(\sum_{l=1}^nc_{li}\omega_{lj}\Big)^2 \Big)=\E\Big(\sum_{l=1}^nc_{li}^2\omega_{lj}^2\Big)= \sum_{l=1}^nc_{li}^2\E(\omega_{lj}^2)= \frac{a^2}{3}\sum_{l=1}^nc_{li}^2= \frac{a^2}{3}.
\]
\end{proof}

Finally, on the expected Frobenius norm of a pre-orthogonalized and scaled uniform random matrix, we have the following result.
\begin{proposition}\label{sgt1}
Let $F\in\R^{m\times k}$, $C\in\R^{n\times k}$ and $G\in\R^{\ell \times q}$be  constant matrices, where $C$ has orthonormal columns. Suppose $\Omega\in\mathbb{R}^{n\times \ell}$ is a uniform random matrix with $\omega_{ij}\in[-\sqrt{3},\sqrt{3}]$. Then $\mathbb{E}(\|FHG\|_{F}^{2}) = \|F\|_{F}^{2}\|G\|_{F}^{2}$ for  $H=C^T\Omega$.
\end{proposition}
\begin{proof}
It is easy to check that $\mathbb{E}(\omega_{lr})=0$ and $\mathbb{E}(\omega_{lr}^2)=1$ for all $1\le l\le n$ and $1\le r\le \ell$, and  $\mathbb{E}(\omega_{lr}\omega_{ab})=\E(\omega_{lr})\E(\omega_{ab})=0$ for all $1\leq l,a \leq n$  and $1\leq r,b\leq \ell$ with $(l,r)\neq(a,b)$. We note that $H=C^T\Omega$.  By Lemma \ref{unif:rotation} we have $\E (h_{ij})=0$ and $\E (h_{ij}^2) = 1$ for all $1\leq i \leq k$ and $1\leq j \leq \ell$.
Thus,
\begin{eqnarray*}
&& \mathbb{E} (\|FHG\|_{F}^{2})= \mathbb{E}\Big(\sum\limits_{i=1}^{m}\sum\limits_{j=1}^{q}\Big(\sum\limits_{l=1}^{k}\sum\limits_{r=1}^{\ell}f_{il}h_{lr}g_{rj}\Big)^{2}\Big)
 =\sum\limits_{i=1}^{m}\sum\limits_{j=1}^{q}\sum\limits_{l=1}^{k}\sum\limits_{r=1}^{\ell}\E\Big(f_{il}^2h_{lr}^2g_{rj}^{2}\Big) \\
&& \quad=\sum\limits_{i=1}^{m}\sum\limits_{j=1}^{q}\sum\limits_{l=1}^{k}\sum\limits_{r=1}^{\ell}\Big(f_{il}^2\E(h_{lr}^2)g_{rj}^{2}\Big) =\sum\limits_{i=1}^{m}\sum\limits_{j=1}^{q}\sum\limits_{l=1}^{k}\sum\limits_{r=1}^{\ell}(f_{il}^2g_{rj}^{2})
=\|F\|_{F}^{2} \|G\|_{F}^{2}.\nonumber
\end{eqnarray*}
\end{proof}

\section{A randomized scheme for  the range approximation via uniform sampling}\label{sec3}
In this section, we present a randomized scheme for  the range approximation via uniform sampling.
In \cite{random},  a randomized  range finder was proposed  as follows: Given a matrix $A\in\Rmn$, an integer $\ell$, draw a random test matrix $\Omega\in\R^{n\times \ell}$,  form the $m\times \ell$ matrix $Y=A\Omega$ and  construct an $m\times \ell$ matrix $Q$ with orthonormal columns, which can be obtained by the QR factorization $Y=QR$. Some tail bounds for the approximation error $\|A-QQ^TA\|_2$ or $\|A-QQ^TA\|_F$ were provided if $\Omega$ is a Gaussian test  matrix
or a SRFT test matrix.

As in \cite{random, random2}, we can find some integer $\ell\le\min\{m,n\}$ and an $m\times \ell$ real orthonormal matrix $Q^{(\ell)}$ such that $\|A-Q^{(\ell)}(Q^{(\ell)})^TA\|_F\le \epsilon$ for some tolerance $\epsilon>0$ by using  a blockwise adaptive  randomized range finder via uniform test matrices, which is described as Algorithm \ref{alg:alg1}.
\begin{algorithm}[ht]
\caption{Blockwise adaptive uniformly randomized range finder}\label{alg:alg1}
{\bf Input:} A matrix $A\in \R^{m\times n}$, a tolerance $\epsilon>0$, and a blocksize $1\le b<n$. Let $s=\lfloor\frac{n}{b}\rfloor$ and $\bff\equiv (f_1,\ldots,f_{s+1})^T=(b,\ldots,b,n-sb)^T\in\R^{s+1}$.\\
{\bf Output:} A matrix $Q$ with orthonormal  columns such that $\|A-QQ^TA\|_F\le\epsilon$ and $Q^TA$ has full row rank  with probability one.\\
\begin{algorithmic}[1]
\STATE $Q_{0}=[]$ and $i=1$.\\
\STATE \textbf{while} $i<s+2$ \\
\STATE Draw a uniform random matrix $\Omega_i\in\R^{n\times f_i}$  with  $\omega_{ij}\in [-\sqrt{3},\sqrt{3}]$. \\
\STATE Form the matrix product $Y_i=A\Omega_i$. \\
\STATE Set $Y_i=(I-Q_{i-1}Q_{i-1}^T)Y_i$. \\
\STATE Compute the reduced QR decomposition $Y_i=P_iR_i$.
\STATE \textbf{if} $\min_{1\le l\le b}|(R_i)_{ll}| > \epsilon$
\STATE Set $Q_{i}=[Q_{i-1},P_i]$.
\STATE  \textbf{else}
\STATE Set $l_i=\min\{l \;|\; |(R_i)_{ll}| \le \epsilon\}$.
\STATE Set $Q_{i}=[Q_{i-1},P_i(:,1:l_i-1)]$ if $l_i>1$ and $Q_{i}=Q_{i-1}$ if $l_i=1$.
\STATE return $Q=Q_i$.
\STATE \textbf{end if}
\STATE $i=i+1$.
\STATE \textbf{if} $n-ib=0$ \textbf{then}
\STATE stop
\STATE \textbf{end if}
\STATE \textbf{end while}
\end{algorithmic}
\end{algorithm}

On  Algorithm \ref{alg:alg1}, we have the following result.
\begin{theorem}
Suppose  $\rank(A)\ge 1$.  For Algorithm  {\rm\ref{alg:alg1}}, there exist an integer $i\ge 1$ and a uniform random matrix $\Omega$ such that   $A\Omega$ admits the reduced QR decomposition $A\Omega=QR$ with $Q$ having orthonormal columns, $\|(I-QQ^T)A\|_F\le \epsilon$, and $Q^TA$ has full row rank  with probability one.
\end{theorem}
\begin{proof}
For simplicity, we assume that $n=tb$. If $\min_{1\le l \le b}|(R_1)_{ll}|\le \epsilon$, then we have $\rank(A\Omega_1)=\min\{\rank(A), b\}\ge 1$ with probability one by Lemma {\rm\ref{t2}} since  $\rank(A)\ge 1$. This means that $l_1>1$. Since $Y_1=A\Omega_1=P_1R_1$, we have $Q=P_1(:,1:l_1-1)$ and $A\Omega_1(:,1:l_1-1)$ admits the QR decomposition:
\[
A\Omega_1(:,1:l_1-1)=Y_1(:,1:l_1-1)=P_1R_1(:,1:l_1-1)=QR_1(1:l_1-1,1:l_1-1).
\]
Using Lemma {\rm\ref{t2}}, we have $l_1-1=\rank(R_1(1:l_1-1,1:l_1-1))=\rank(Q^TA\Omega_1(:,1:l_1-1)=\min\{\rank(Q^TA),l_1-1\}$ with probability one and thus $Q^TA$ is of full row rank  with probability one. It follows from the classical Gram-Schmidt algorithm for $A\Omega_1$ that
$(R_1)_{l_1l_1}=\|A(\Omega_1(:,l_1)-QQ^TA\Omega_1(:,l_1)\|_2$.
By Proposition \ref{sgt1} we have
$\|(I-QQ^T)A\|_F^{2}=\mathbb{E}(\|(I-QQ^T)A(\Omega_1(:,l_1)\|_2^{2})=\E((R_1)_{l_1l_1}^2) \le \epsilon^2$.

In the following, we assume that  $\min_{1\le l \le b}|(R_i)_{ll}|\le \epsilon$ for some $i>1$.
By hypothesis, $\min_{1\le l \le b}|(R_1)_{ll}|>\epsilon$. Then  we have  $Y_1=A\Omega_1=P_1R_1$. It follows  that $A\Omega_1$ admits the QR decomposition  $A\Omega_1=P_1R_1$, where $P_1$ has orthonormal columns.

Analogously, by hypothesis,  for any $1\le j\le i-1$, $\min_{1\le l \le b}|(R_j)_{ll}|>\epsilon$.
Then  $Y_{j}=(I_m-Q_{j-1}Q_{j-1}^T)A\Omega_{j} =P_{j}R_{j}$, where $Q_{j-1}=[Q_{j-2},P_{j-1}]=[P_1,\ldots,P_{j-1}]$. It follows  that
$A\Omega_j=P_1P_1^TA\Omega_j+\cdots+P_{j-1}P_{j-1}^TA\Omega_j+P_jR_j$ and $P_{j-1}^TP_j=0$.
This indicates that $A[\Omega_1, \Omega_2,\ldots, \Omega_j]$ admits the QR decomposition:
\[
A[\Omega_1,  \Omega_2,\ldots, \Omega_j]=[P_1,P_2,\ldots,P_j]\left[\begin{array}{cccc}
R_1 &R_{12} & \cdots & R_{1j}\\
       & R_2     & \cdots & R_{2j}\\
       &             & \ddots  & \vdots \\
       &             &   & R_j\\
\end{array}
\right],
\]
where $R_{ab}=P_a^TA\Omega_b$ for all $1\le a<b\le j$.

By hypothesis, $\min_{1\le l \le b}|(R_i)_{ll}|\le\epsilon$.  If $l_i=1$, then $Q=Q_{i-1}$. From the above analysis, we know that $A[\Omega_1, \Omega_2,\ldots, \Omega_{i-1}]$ admits the QR decomposition:
\[
A[\Omega_1,  \Omega_2,\ldots, \Omega_{i-1}]=Q\left[\begin{array}{cccc}
R_1 &R_{12} & \cdots & R_{1,i-1}\\
       & R_2     & \cdots & R_{2,i-1}\\
       &             & \ddots  & \vdots \\
       &             &   & R_{i-1}\\
\end{array}
\right].
\]
Using Lemma {\rm\ref{t2}}, we have $(i-1)b=\rank(Q^TA[\Omega_1, \Omega_2,\ldots, \Omega_{i-1}])=\min\{\rank(Q^TA),$ $(i-1)b\}$ with probability one and thus $Q^TA$ is of full row rank  with probability one.
We note that $Y_{i}=(I_m-Q_{i-1}Q_{i-1}^T)A\Omega_{i} =P_{i}R_{i}$.
It follows from the classical Gram-Schmidt algorithm for  $Y_{i}$ that
$(R_i)_{11}=\|A\Omega_i(:,1)-QQ^TA\Omega_i(:,1)\|_2$.
By Proposition  \ref{sgt1} we have
$\|(I-QQ^T)A\|_F^{2}=\mathbb{E} (\|(I-QQ^T)A(\Omega_i(:,1)\|_2^{2})=\E ((R_i)_{11}^2) \le \epsilon^2$.

If $l_i>1$, then we have $Y_i(:,1:l_i-1)=(I_m-Q_{i-1}Q_{i-1}^T)A\Omega_{i}(:,1:l_i-1) =P_{i}R_{i}(:,1:l_i-1)=P_{i}(:,1:l_i-1)R_{i}(1:l_i-1,1:l_i-1)$ and $P_{i-1}^TP_i(:,1:l_i-1)=0$. This indicates that $Q$ has orthonormal columns. Thus,  $A[\Omega_1, \Omega_2,\ldots, \Omega_{i-1}, \Omega_{i}(:,1:l_i-1)]$ admits the QR decomposition
\[
\begin{array}{ll}
&A[\Omega_1,  \Omega_2,\ldots, \Omega_{i-1},\Omega_{i}(:,1:l_i-1)]\\
& \quad= Q\left[\begin{array}{ccccc}
R_1 &R_{12} & \cdots & R_{1,i-1} &P_1^TA\Omega_{i}(:,1:l_i-1)  \\
       & R_2     & \cdots & R_{2,i-1} & P_2^TA\Omega_{i}(:,1:l_i-1)  \\
       &             & \ddots  & \vdots & \vdots   \\
       &             &   & R_{i-1} & \\
       &             &   & &  R_{i}(1:l_i-1,1:l_i-1)\\
\end{array}
\right].
\end{array}
\]
Using Lemma {\rm\ref{t2}}, we have $(i-1)b+l_i-1=\rank(Q^TA[\Omega_1, \Omega_2,\ldots, \Omega_{i-1},\Omega_{i}(:,1:l_i-1))]=\min\{\rank(Q^TA),$ $(i-1)b+l_i-1\}$ with probability one and thus $Q^TA$ is of full row rank  with probability one.
From the  classical Gram-Schmidt algorithm for  $Y_{i}$ we have
$(R_i)_{l_il_i}=\|A\Omega_i(:,l_i)-QQ^TA\Omega_i(:,l_i)\|_2$.
By Proposition  \ref{sgt1} we have
$\|(I-QQ^T)A\|_F^{2}=\mathbb{E} (\|(I-QQ^T)A(\Omega_i(:,l_i)\|_2^{2}) =\E ((R_i)_{l_il_i}^2) \le \epsilon^2$.
\end{proof}

\section{A GSVD algorithm of a full-rank GMP}\label{sec4}
In this section, we provide a new GSVD algorithm for a Grassman matrix pair (GMP) $\{A,L\}$, where  $A$ is assumed to be of full rank.  We first recall the definition of GMP.
\begin{definition} \cite{sun1983}
Let $A \in \R^{m\times n}$ and $L \in \R^{p\times n}$. A matrix pair $\{A,L\}$ is an $(m,p,n)$-Grassman matrix pair (GMP) if $\rank([A^T,L^T]^T)=n$.
 \end{definition}

We also recall the  definition of the GSVD of a GMP.
\begin{definition}\cite{sun1983,CS}\label{def-gsvd}
For a given $(m,p,n)$-GMP $\{A,L\}$, there exist an orthogonal matrix $U\in\Rmm$, an orthogonal matrix $V\in\Rnn$ and a nonsingular matrix $X\in\Rnn$ such that
\begin{eqnarray} \label{def}
& U^TAX=\Sigma_A\quad\mbox{and}\quad  V^TLX=\Sigma_L, & \\
& \Sigma_A=
\begin{blockarray}{ccc}
n-r-s & r+s  & \\
\begin{block}{[cc]c}
0 &  & m-r-s \\
& D_1  & r+s\\
\end{block}
\end{blockarray},
\quad \Sigma_L=
\begin{blockarray}{ccc}
n-r & r  & \\
\begin{block}{[cc]c}
D_2 & & n-r\\
& 0  & p+r-n \\
\end{block}
\end{blockarray},& \nonumber 
\end{eqnarray}
where  $D_1=\diag(\alpha_{n-r-s+1},\ldots,\alpha_{n})$ and $D_2=\diag(\beta_{1},\ldots,\beta_{n-r})$  with
\begin{align*}
0=\alpha_1=\cdots=\alpha_{n-r-s}<\alpha_{n-r-s+1}\le \cdots\le \alpha_{n-r}<\alpha_{n-r+1}=\cdots=\alpha_n=1, \\
1=\beta_1=\cdots=\beta_{n-r-s}>\beta_{n-r-s+1}\ge \cdots\ge  \beta_{n-r}>\beta_{n-r+1}=\cdots=\beta_n= 0,
\end{align*}
and $\alpha_i^2+\beta_i^2=1$ for $i=1,\ldots,n$.
\end{definition}

In the following, we give a new algorithm for computing the GSVD of an $(m,p,n)$-GMP $\{A,L\}$, where  $A$ has full rank. Let $[A^T,L^T]^T$ admit the reduced QR decomposition:
\[\left[
\begin{array}{c}
A \\
L 
\end{array}
\right] =\widetilde{Q}\widetilde{R}\equiv \left[
\begin{array}{c}
\widetilde{Q}_1 \\
\widetilde{Q}_2 
\end{array}
\right]\widetilde{R},\quad \widetilde{Q}_1 \in \R^{m\times n},
\]
where $\widetilde{Q}\in\R^{(m+p)\times n}$ has orthonormal columns and $\widetilde{R} \in \Rnn$ is a nonsingular upper triangular matrix. Then we have $A=\widetilde{Q}_1\widetilde{R}$ and $L=\widetilde{Q}_2\widetilde{R}$ and $\rank(A)=\rank(\widetilde{Q}_1)$ and
$\rank(L)=\rank(\widetilde{Q}_2)$.
If $\rank(A)=n\le m$, then $\widetilde{Q}_1^T\widetilde{Q}_1$ is symmetric and positive definite. Let $Q_1^TQ_1$ admit the symmetric Schur decomposition  $\widetilde{Q}_1^T\widetilde{Q}_1 = S\Psi S^T$, where $S\in\Rnn$ is orthogonal  and $\Psi=\diag(\psi_1,\ldots,\psi_n)\in\Rnn$ is a diagonal matrix with $0<\psi_1\le \cdots\le \psi_{n-r}<\psi_{n-r+1}= \cdots=\psi_n=1$. Let $D_1=\diag(\sqrt{\psi_1},\ldots, \sqrt{\psi_n})$ and $D_2=\diag(\sqrt{1-\psi_{1}},\ldots, \sqrt{1-\psi_{n-r}})$. We note that  $\widetilde{Q}_1^T\widetilde{Q}_1+\widetilde{Q}_{2}^T\widetilde{Q}_{2}=I_n$. Then
\[
\widetilde{Q}_{2}^T\widetilde{Q}_{2}=I_n-\widetilde{Q}_{1}^T\widetilde{Q}_{1}=I_n-SD_1^2S^T = S_1D_2^2S_1^T,\quad S_1=S(:,1:n-r).
\]
Let $U_2= \widetilde{Q}_{1}SD_1^{-1}$, $V_1 = \widetilde{Q}_{2}S_1 D_2^{-1}$ and $X = \widetilde{R}^{-1}S$ with $X_1=\widetilde{R}^{-1}S_1$. Then both $U_2$ and $V_1$ have orthonormal columns. Hence, the GSVD of the $(m,p,n)$-GMP $\{A,L\}$ is given by
\begin{align*}
&U_2^TA X =D_1^{-1}S^T\widetilde{Q}_{1}^T\widetilde{Q}_1\widetilde{R}\widetilde{R}^{-1}S = D_1^{-1}S^TSD_1^2S^TS=D_1,\\
&V_1^TLX_1 = D_2^{-1}S_1^T\widetilde{Q}_2^T\widetilde{Q}_2\widetilde{R}\widetilde{R}^{-1}S_1 = D_2^{-1}S_1^TS_1D_2^2S_1^TS_1=D_2.
\end{align*}

If $\rank(A)=m\le n$, then $\widetilde{Q}_1^T\widetilde{Q}_1$ is symmetric and positive semidefinite. Let $\widetilde{Q}_1^T\widetilde{Q}_1$ admit the symmetric Schur decomposition  $\widetilde{Q}_1^T\widetilde{Q}_1 = S\Psi S^T$, where $S\in\Rnn$ is orthogonal  and $\Psi=\diag(\psi_1,\ldots,\psi_n)\in\Rnn$ is a diagonal matrix with $0=\psi_1=\cdots=\psi_{n-m}<\psi_{n-m+1}\le \cdots\le\psi_{n-r}<\psi_{n-r+1}= \cdots=\psi_n=1$. Let $D_1=\diag(\sqrt{\psi_{n-m+1}},\ldots, \sqrt{\psi_n})$ and $D_2=\diag(\sqrt{1-\psi_{1}},\ldots, \sqrt{1-\psi_{n-r}})$. We note that  $\widetilde{Q}_{1}^T\widetilde{Q}_{1}+\widetilde{Q}_{2}^T\widetilde{Q}_{2}=I_n$. Then
\[
\widetilde{Q}_{2}^T\widetilde{Q}_{2}=I_n-\widetilde{Q}_{1}^T\widetilde{Q}_{1}=I_n-S\Psi S^T = S_1D_2^2S_1^T,\quad S_1=S(:,1:n-r).
\]
Let $U = \widetilde{Q}_{1}S_2D_1^{-1}$ with $S_2=S(:,n-m+1:n)$, $V_1 = \widetilde{Q}_{2}S_1 D_2^{-1}$ and $X = \widetilde{R}^{-1}S$ with $X_2=\widetilde{R}^{-1}S_2$ and $X_1=\widetilde{R}^{-1}S_1$. Then $U$ is orthogonal  and $V_1$ has orthonormal columns. Hence, the GSVD of the $(m,p,n)$-GMP $\{A,L\}$ is given by
\begin{align*}
&U^TA X_2 = D_1^{-1}S_2^T\widetilde{Q}_{1}^T\widetilde{Q}_1\widetilde{R}\widetilde{R}^{-1}S_2 =D_1^{-1}S_2^TS_2D_1^2S_2^TS_2=D_1,\\
&V_1^TLX_1 = D_2^{-1}S_1^T\widetilde{Q}_2^T\widetilde{Q}_2\widetilde{R}\widetilde{R}^{-1}S_1 = D_2^{-1}S_1^TS_1D_2^2S_1^TS_1=D_2.
\end{align*}

Based on the above analysis, an algorithm can be obtained for computing the GSVD of  a full-rank GMP, which is described as Algorithm \ref{alg:alg2}.

\begin{algorithm}[h]
\caption{Fast GSVD of a full-rank GMP} \label{alg:alg2}
{\bf Input:} A full-rank $(m,p,n)$-GMP $\{A,L\}$ with $A$ having full rank.\\
{\bf Output:} For $\rank(A)=n$: The matrices $U_2\in\Rmn$ and $V_1\in\R^{p\times (n-r)}$ with orthonormal columns, the nonsingular matrix $X\in\Rnn$,  and the diagonal matrices $D_1$ and $D_2$ with positive diagonal entries such that $U_2^TAX=D_1$ and $V_1^TLX_1=D_2$,  where $D_1$ and $D_2$ are nonsingular diagonal matrices defined as in \eqref{def} with $s=n-r$,  $X_1=X(:,1:n-r)$;  For $\rank(A)=m$: The matrices $U\in\Rmm$ and $V_1\in\R^{p\times (n-r)}$ with orthonormal columns, the nonsingular matrix $X\in\Rnn$,  and the diagonal matrices $D_1$ and $D_2$ such that $U^TAX_2=D_1$ and $V_1^TLX_1=D_2$, where $D_1$ and $D_2$ are nonsingular diagonal matrices defined as in \eqref{def} with $s=m-r$,  $X_1=X(:,1:n-r)$, and $X_2=X(:,n-m+1:n)$.
\begin{algorithmic}[1]
\STATE Compute the reduced QR decomposition of $[A^T,L^T]^T=\widetilde{Q}\widetilde{R}$, where $\widetilde{Q}=[\widetilde{Q}_1^T,\widetilde{Q}_2^T]^T\in\R^{(m+p)\times n}$ has orthonormal columns with $\widetilde{Q}_1 \in \Rmn$ and $\widetilde{R} \in \Rnn$  is upper triangular.\\
\IF{$n \leq m$}
    \STATE Compute the symmetric Schur decomposition $\widetilde{Q}_1^T\widetilde{Q}_1=S\Psi S^T$, where $S\in\Rnn$ is orthogonal  and $\Psi=\diag(\psi_1,\ldots,\psi_n)\in\Rnn$ with $0<\psi_1\le \cdots\le \psi_{n-r}<\psi_{n-r+1}= \cdots=\psi_n=1$.
    \STATE Set $D_1=\diag(\sqrt{\psi_1},\ldots, \sqrt{\psi_n})$ and
    $D_2=\diag(\sqrt{1-\psi_{1}},\ldots, \sqrt{1-\psi_{n-r}})$.
     \STATE Set $U_2 = \widetilde{Q}_{1}SD_1^{-1}$, $V_1 = \widetilde{Q}_{2}S_1 D_2^{-1}$ with $S_1=S(:,1:n-r)$, and $X = \widetilde{R}^{-1}S$ with $X_1=\widetilde{R}^{-1}S_1$.
    \ELSE
     \STATE Compute the symmetric Schur decomposition $\widetilde{Q}_1^T\widetilde{Q}_1=S\Psi S^T$, where $S\in\Rnn$ is orthogonal  and $\Psi=\diag(\psi_1,\ldots,\psi_n)\in\Rnn$ with $0=\psi_1=\cdots=\psi_{n-m}<\psi_{n-m+1}\le \cdots\le\psi_{n-r}<\psi_{n-r+1}= \cdots=\psi_n=1$.
    \STATE Set $D_1=\diag(\sqrt{\psi_{n-m+1}},\ldots, \sqrt{\psi_n})$ and $D_2=\diag(\sqrt{1-\psi_{1}},\ldots, \sqrt{1-\psi_{n-r}})$.
    \STATE Set $U = Q_{1}S_2D_1^{-1}$ with $S_2=S(:,n-m+1:n)$, $V_1 = Q_{2}S_1 D_2^{-1}$ with  $S_1=S(:,1:n-r)$, and $X = \widetilde{R}^{-1}S$ with $X_2=\widetilde{R}^{-1}S_2$ and $X_1=\widetilde{R}^{-1}S_1$.
\ENDIF
\end{algorithmic}
\end{algorithm}

\section{Tikhonov regularization via  two-sided uniformly randomized GSVD}\label{sec5}
In this section, we propose a   two-sided uniformly randomized GSVD algorithm for solving the large-scale  Tikhonov regularization problem \eqref{tikhonov}. Suppose that $\mathcal{N}(A) \cap \mathcal{N}(L)=\{\bf 0\}$, i.e., $\rank([A^T, L^T]^T)=n$. Then problem \eqref{tikhonov} has a unique  regularized solution \eqref{tsol}.  This shows that $\{A,L\}$ is a GMP.
By Definition \ref{def-gsvd},   the GMP $\{A,L\}$ admits the GSVD defined as \eqref{def}.
Then problem \eqref{tikhonov} is reduced to \begin{align}
\underset{\by\in \mathbb{R}^n}{\min}{ \|\Sigma_A \by-U^T\bb\|^2+\lambda^2\|\Sigma_L\by\|^2}, \quad \by=X^{-1}\bx. \label{exact}
\end{align}
Let $U=[\bu_1,\ldots,\bu_m]$. The unique regularized solution $\by_\lambda\in\Rn$ of problem \eqref{exact} is given by $\by_{\lambda}=([\Sigma_A^T, \Sigma_L^T]^T)^\dag[\bb^T,{\bf 0}^T]^T$, i.e.,
\[
\by_{\lambda}=\Big(0,\ldots, 0,\frac{\alpha_{n-r-s+1}\eta_{n-r-s+1}}{\alpha_{n-r-s+1}^2+\lambda^2\beta_{n-r-s+1}^2}, \ldots, \frac{\alpha_{n-r}\eta_{n-r}}{\alpha_{n-r}^2+\lambda^2\beta_{n-r}^2},\eta_{n-r+1},\ldots,\eta_n\Big)^T,
\]
where $\eta_i=\bu_{m-n+i}^T\bb$ for all $1\le i\le n$.
Let $X=[\bx_1,\ldots,\bx_n]$. Then the unique regularized solution of problem \eqref{tikhonov} is given by
\[
\bx_\lambda=\sum_{i=n-r-s+1}^{n-r}\frac{\alpha_i} {\alpha_i^2+\lambda^2\beta_i^2}\eta_i\bx_i+ \sum_{i=n-r+1}^{n} \eta_i\bx_i.
\]

When problem  \eqref{tikhonov}  is large-scale, calculating the GSVD of the GMP $\{A,L\}$ requires a large amount of computation and storage.  Then, the  GSVD and truncated GSVD are not suitable for solving large-scale problem  \eqref{tikhonov}. 

Recently, some randomized methods were proposed for solving the large-scale problem  \eqref{tikhonov} \cite{random, jia2018, 2016, XZ13,xiang2015}.   
These randomized algorithms only adopt the one-sided sampling based on the Gaussian random matrix while  only the number of rows or columns is reduced and the reduced problem may be still large-scale and ill-posed. 

In the following,  we propose a two-sided randomized GSVD algorithm based on the uniform random test matrix for the large-scale problem  \eqref{tikhonov} with both overdetermined and underdetermined  cases.  
\subsection{Overdetermined case ($m\ge n$)}
For a prescribed tolerance $\epsilon>0$, by Algorithm \ref{alg:alg1}, there exists a uniform random test matrix $\Omega_P\in\R^{n\times\ell_1}$ such that $A\Omega_P$ admits the QR decomposition $A\Omega_P=PR$  with  $P\in\R^{m\times\ell_1}$ having orthonormal columns, $\|A- PP^TA\|_F\le\epsilon$  and  $P^TA$ has full row rank  with probability one, where $\ell_1\le n\le m$. Then $P^TA=(A\Omega_P R^{-1})^TA$. By hypothesis, $\mathcal{N}(A) \cap \mathcal{N}(L)=\{\bf 0\}$.  As in \cite{2016}, it is easy to check that $\mathcal{N}(P^TA) \cap \mathcal{N}(L)=\{\bf 0\}$. Then, when $n\ll m$, one may use the GSVD of the small-scale GMP $\{P^TA,L\}$ to compute an approximate GSVD of the GMP $\{A,L\}$. 

For some applications, both $m$ and $n$ may be very large and thus the GSVD of $\{P^TA,L\}$ is still expensive. Then we can further use Algorithm \ref{alg:alg1} to $A^TP$ to find an approximate basis matrix $Q \in \mathbb{R}^{n\times \ell_2}$ with  orthonormal columns  for the range of $A^TP$ such that $\|P^TA- P^TAQQ^T\|_F=\|A^TP- QQ^TA^TP\|_F\le\epsilon$ and $Q^T(A^TP)$ has full row rank  with probability one, where $\ell_2\le\min\{m,\ell_1\}=\ell_1$.  Since $\mathcal{N}(P^TA) \cap \mathcal{N}(L)=\{\bf 0\}$, we have $\mathcal{N}(P^TAQ) \cap \mathcal{N}(LQ)=\{\bf 0\}$, i.e., $\rank((P^TAQ)^T,(LQ)^T]^T)=\ell_2$.  Suppose that $\ell_2\le\ell_1\ll n$. Then, we can use the GSVD of the  GMP $\{P^TAQ,LQ\}$ to compute an approximate GSVD of the GMP $\{A,L\}$.  We note that $P^TAQ$ has full column rank  with probability one. Then,  using Algorithm \ref{alg:alg2}, $[(P^TAQ)^T,(LQ)^T]^T$ admits the  following GSVD:
\[
\left[
\begin{array}{c}
P^TAQ\\
LQ 
\end{array}
\right] =\left[
\begin{array}{cc}
\widetilde{U}_2 & \\
& V_1
\end{array}
\right] \left[
\begin{array}{c}
D_1 W\\
D_2 W(1:\ell_2-r,:)
\end{array}
\right],
\]
where $\widetilde{U}_2\in\R^{\ell_1\times \ell_2}$ and $V_1\in\R^{p\times (\ell_2-r)}$ have orthonormal columns, $D_1$ and $D_2$ are diagonal matrices with positive diagonal entries, and $W=\widetilde{X}^{-1}\in\R^{\ell_2\times\ell_2}$ is nonsingular. Hence, we obtain the following approximate GSVD of $[A^T,L^T]^T$:
\BE\label{appr:alo}
\left[
\begin{array}{c}
A\\
L 
\end{array}
\right]\approx
\left[
\begin{array}{c}
PP^TAQQ^T\\
LQQ^T \\
\end{array}
\right]=\left[
\begin{array}{cc}
U_2 & \\
& V_1
\end{array}
\right] \left[
\begin{array}{c}
D_1Z \\
D_2 Z(1:\ell_2-r,:)
\end{array}
\right],
\EE
where $U_2=P\widetilde{U}_2\in\R^{m\times \ell_2}$ has orthonormal columns and $Z=WQ^T=\widetilde{X}^{-1}Q^T\in\R^{\ell_2\times n}$ has full row rank.
The corresponding two-sided uniform randomized GSVD algorithm for the case of $m\ge n$ can be described as Algorithm  \ref{alg:alg4}.
\begin{algorithm}[ht]
\caption{Two-sided uniform randomized GSVD ($m\geq n$)}\label{alg:alg4}
{\bf Input:} $A\in \mathbb{R}^{m\times n}$, $L\in \mathbb{R}^{p\times n}$, and a tolerance $\epsilon>0$. \\
{\bf Output:} The matrices $U_2\in\R^{m\times \ell_2}$ and $V_1\in\R^{p\times (\ell_2-r)}$ with orthonormal columns, the diagonal matrices $D_1\in\R^{\ell_2\times\ell_2}$ and $D_2\in\R^{(\ell_2-r)\times (\ell_2-r)}$, and the full row rank matrix $Z\in\R^{\ell_2\times n}$.
\begin{algorithmic}[1]
\STATE Compute an  approximate basis matrix $P\in\R^{m\times\ell_1}$ with orthonormal columns for the range of $A$ such that  $\|PP^{T}A-A\|\leq\epsilon$ via Algorithm \ref{alg:alg1} with a blocksize $1\le b<n$ and  the uniform random test matrix $\Omega_P\in\R^{n\times \ell_1}$. \\
\STATE Compute an  approximate basis matrix $Q\in\R^{n\times\ell_2}$ with orthonormal columns for the range of $A^TP$ such that $\|A^TP-QQ^{T}A^TP\|\leq\epsilon$ via Algorithm \ref{alg:alg1} with a blocksize $1\le b<\ell_1$ and the uniform random test matrix  $\Omega_Q\in\R^{\ell_1\times \ell_2}$.
\STATE Compute the GSVD of the GMP $\{P^{T}AQ,LQ\}$ via Algorithm \ref{alg:alg2}: 
$
\left[
\begin{array}{c}
P^TAQ\\
LQ 
\end{array}
\right] =\left[
\begin{array}{cc}
\widetilde{U}_2 & \\
& V_1
\end{array}
\right] \left[
\begin{array}{c}
D_1 W\\
D_2 W(1:\ell_2-r,:)
\end{array}
\right]
$, where $W=\widetilde{X}^{-1}\in\R^{\ell_2\times\ell_2}$.\\
\STATE Form the $m\times \ell_2$ matrix $U_2=P\widetilde{U}_2$ and  the $\ell_2\times n$ matrix $Z=\widetilde{X}^{-1}Q^T$. \\
\end{algorithmic}
\end{algorithm}

Therefore, the regularized solution of problem  \eqref{tikhonov} can be approximated by
\begin{flalign*} \hspace{15pt}
\begin{split}
\bx_{\lambda}&= \left[
\begin{array}{c}
A\\
\lambda L 
\end{array}
\right]^\dag
\left[
\begin{array}{c}
\bb\\
\bf 0
\end{array}
\right] \approx
\left[
\begin{array}{c}
PP^TAQQ^T\\
\lambda LQQ^T \\
\end{array}
\right]^\dag \left[
\begin{array}{c}
\bb\\
\bf 0
\end{array}
\right] =
Q\left[
\begin{array}{c}
PP^TAQ\\
\lambda LQ \\
\end{array}
\right]^\dag
\left[
\begin{array}{c}
\bb\\
\bf 0
\end{array}
\right] \\
&=
Q\left( \left[
\begin{array}{c}
PP^TAQ\\
\lambda LQ \\
\end{array}
\right]^T
\left[
\begin{array}{c}
PP^TAQ\\
\lambda LQ \\
\end{array}
\right]\right)^{-1}
 \left[
\begin{array}{c}
PP^TAQ\\
\lambda LQ \\
\end{array}
\right]^T
\left[
\begin{array}{c}
\bb\\
\bf 0
\end{array}
\right] 
\end{split}&
\end{flalign*}
\begin{flalign*}  \hspace{15pt}
\begin{split}
&=
Q\left( \left[
\begin{array}{c}
P^TAQ\\
\lambda LQ \\
\end{array}
\right]^T
\left[
\begin{array}{c}
P^TAQ\\
\lambda LQ \\
\end{array}
\right]\right)^{-1}
\left[
\begin{array}{c}
P^TAQ\\
\lambda LQ \\
\end{array}
\right]^T
\left[
\begin{array}{c}
P^T\bb\\
\bf 0
\end{array}
\right] \\
&=
Q\left[
\begin{array}{c}
P^TAQ\\
\lambda LQ \\
\end{array}
\right]^\dag
\left[
\begin{array}{c}
P^T\bb\\
\bf 0
\end{array}
\right].
\end{split}&
\end{flalign*}
This shows that  $\bx_{\lambda}$ can be approximated by the  randomized GSVD in Algorithm  \ref{alg:alg4} for the case of $m\ge n$.
\subsection{Underdetermined case ($m< n$)}
For a prescribed tolerance $\epsilon>0$, by Algorithm \ref{alg:alg1} with a series of uniform random test matrices, we first find an approximate basis matrix $Q \in \mathbb{R}^{n\times \ell_1}$ with  orthonormal columns  for the range of $A^T$ such that $\|A- AQQ^T\|_F=\|A^T- QQ^TA^T\|_F\le\epsilon$ and  $Q^TA^T$ has full row rank  with probability one, where $\ell_1\le m< n$. By hypothesis, $\mathcal{N}(A) \cap \mathcal{N}(L)=\{\bf 0\}$. Thus, $\mathcal{N}(AQ) \cap \mathcal{N}(LQ)=\{\bf 0\}$. Then, one may use the GSVD of the GMP $\{AQ,LQ\}$ to compute an approximate GSVD of the GMP $\{A,L\}$. 

Since we aim to solve the large-scale Tikhonov regularization problem  \eqref{tikhonov}, it may be still very expensive to compute the GSVD of the GMP $\{AQ,LQ\}$. Then we can further employ  Algorithm \ref{alg:alg1} to $AQ$ to find an approximate basis matrix $P \in \mathbb{R}^{m\times \ell_2}$ with  orthonormal columns  for the range of $AQ$ such that $\|AQ- PP^TAQ\|_F\le\epsilon$ and $P^TAQ$ has full row rank  with probability one, where $\ell_2\le\min\{m,\ell_1\}=\ell_1$.
As in \cite{2016}, it is easy to check that $\mathcal{N}(P^TAQ) \cap \mathcal{N}(LQ)=\{\bf 0\}$ since $\mathcal{N}(AQ) \cap \mathcal{N}(LQ)=\{\bf 0\}$. Then, we can use the GSVD of the   GMP $\{P^TAQ,LQ\}$ to compute an approximate GSVD of the GMP $\{A,L\}$.  We note that $P^TAQ$ has full row rank  with probability one.  Then,  using Algorithm \ref{alg:alg2},  $[(P^TAQ)^T,(LQ)^T]^T$ admits the  following GSVD:
\[
\left[
\begin{array}{c}
P^TAQ\\
LQ 
\end{array}
\right] =\left[
\begin{array}{cc}
\widetilde{U} & \\
& V_1
\end{array}
\right] \left[
\begin{array}{c}
D_1 W(\ell_1-\ell_2+1:\ell_1,:) \\
D_2 W(1:\ell_1-r,:)
\end{array}
\right],
\]
where $\widetilde{U}\in\R^{\ell_2\times \ell_2}$ and $V_1\in\R^{p\times (\ell_1-r)}$ have orthonormal columns, $D_1$ and $D_2$ are diagonal matrices with positive diagonal entries, and $W=\widetilde{X}^{-1}\in\R^{\ell_1\times\ell_1}$ is nonsingular. Hence, we obtain the following approximate GSVD of $[A^T,L^T]^T$:
\BE\label{appr:alu}
\left[
\begin{array}{c}
A\\
L 
\end{array}
\right]\approx
\left[
\begin{array}{c}
PP^TAQQ^T\\
LQQ^T \\
\end{array}
\right]=\left[
\begin{array}{cc}
U_2 & \\
& V_1
\end{array}
\right] \left[
\begin{array}{c}
D_1Z(\ell_1-\ell_2+1:\ell_1,:)  \\
D_2 Z(1:\ell_1-r,:)
\end{array}
\right],
\EE
where $U_2=P\widetilde{U}\in\R^{m\times \ell_2}$ has orthonormal columns and $Z=\widetilde{X}^{-1}Q^T\in\R^{\ell_1\times n}$ has full row rank.
The corresponding two-sided uniform randomized GSVD algorithm for the case of $m< n$ can be described as Algorithm  \ref{alg:alg5}.
\begin{algorithm}[h]
\caption{Two-sided uniform randomized GSVD ($m< n$)}\label{alg:alg5}
{\bf Input:} $A\in \mathbb{R}^{m\times n}$, $L\in \mathbb{R}^{p\times n}$, and a tolerance $\epsilon>0$. \\
{\bf Output:} The matrices $U_2\in\R^{m\times \ell_2}$ and $V_1\in\R^{p\times (\ell_1-r)}$ with orthonormal columns, the diagonal matrices $D_1\in\R^{\ell_2\times\ell_2}$ and $D_2\in\R^{(\ell_1-r)\times (\ell_1-r)}$, and the full row rank matrix $Z\in\R^{\ell_1\times n}$.
\begin{algorithmic}[1]
\STATE Compute an  approximate basis matrix $Q\in\R^{n\times\ell_1}$ with orthonormal columns for the range of $A^T$ such that $\|A^T-QQ^TA^T\|\leq\epsilon$ via Algorithm \ref{alg:alg1} with a blocksize $1\le b<m$ and  the uniform random test matrix $\Omega_Q\in\R^{m\times\ell_1}$.  \\
\STATE Compute an  approximate basis matrix $P\in\R^{m\times\ell_2}$ with orthonormal columns for the range of $AQ$ such that $\|PP^{T}AQ-AQ\|\leq\epsilon$ via Algorithm \ref{alg:alg1} with a blocksize $1\le b<\ell_1$ and the uniform random test matrix $\Omega_P \in \R^{\ell_1\times\ell_2}$.
\STATE Compute the GSVD of the GMP $\{P^{T}AQ,LQ\}$ via Algorithm \ref{alg:alg2}: 
$
\left[
\begin{array}{c}
P^TAQ\\
LQ 
\end{array}
\right] =\left[
\begin{array}{cc}
\widetilde{U} & \\
& V_1
\end{array}
\right] \left[
\begin{array}{c}
D_1 W(\ell_1-\ell_2+1:\ell_1,:)\\
D_2 W(1:\ell_1-r,:)
\end{array}
\right]
$, where $W=\widetilde{X}^{-1}\in\R^{\ell_1\times\ell_1}$.\\
\STATE Form the $m\times \ell_2$ matrix $U_2=P\widetilde{U}_2$ and  the $\ell_1\times n$ matrix $Z=\widetilde{X}^{-1}Q^T$. \\
\end{algorithmic}
\end{algorithm}

In this case, the regularized solution of problem  \eqref{tikhonov} can be approximated by
\begin{flalign*} \hspace{15pt}
\begin{split}
\bx_{\lambda}&= \left[
\begin{array}{c}
A\\
\lambda L 
\end{array}
\right]^\dag
\left[
\begin{array}{c}
\bb\\
\bf 0
\end{array}
\right] \approx
\left[
\begin{array}{c}
PP^TAQQ^T\\
\lambda LQQ^T \\
\end{array}
\right]^\dag \left[
\begin{array}{c}
\bb\\
\bf 0
\end{array}
\right] =
Q\left[
\begin{array}{c}
PP^TAQ\\
LQ \\
\end{array}
\right]^\dag
\left[
\begin{array}{c}
\bb\\
\bf 0
\end{array}
\right] \notag\\
&=
Q\left( \left[
\begin{array}{c}
P^TAQ\\
LQ \\
\end{array}
\right]^T
\left[
\begin{array}{c}
P^TAQ\\
LQ \\
\end{array}
\right]\right)^{-1}
\left[
\begin{array}{c}
P^TAQ\\
LQ \\
\end{array}
\right]^T
\left[
\begin{array}{c}
P^T\bb\\
\bf 0
\end{array}
\right] 
\end{split}&
\end{flalign*}
\begin{flalign*}  \hspace{15pt}
\begin{split}
&=
Q\left[
\begin{array}{c}
P^TAQ\\
LQ \\
\end{array}
\right]^\dag
\left[
\begin{array}{c}
P^T\bb\\
\bf 0
\end{array}
\right].
\end{split}&
\end{flalign*}
This shows that  $\bx_{\lambda}$ can be approximated by the  randomized GSVD in Algorithm  \ref{alg:alg5} for the case of $m< n$.
\subsection{Choice of  the regularization parameter $\lambda$}
We can choose the regularization parameter $\lambda$ by combining the generalized cross-validation (GCV) method \cite{GCV} or L-curve rule \cite{H92} with the proposed two-sided uniformly randomized GSVD of  the GMP $\{A,L\}$. In our later numerical experiments,  the approximate GSVD  \eqref{appr:alo} or \eqref{appr:alu} is employed for selecting  the parameter $\lambda$ as in  \cite[Algorithm 4]{2016}.
\subsection{Error Analysis}
In this subsection, we give the error analysis for Algorithms \ref{alg:alg4}--\ref{alg:alg5}.
We first review some preliminary results on the perturbation results for Moore-Penrose  pseudo-inverses.

\begin{lemma}[\cite{s1977}, Theorem 3.3]\label{ab}
For any $G_1,G_2\in\Cmn$, one has
\[
\|G_1^{\dagger}-G_2^{\dagger}\|\leq(1+\sqrt{5})/2\max\{\|G_1^{\dagger}\|^{2},\|G_2^{\dagger}\|^{2}\}\|G_1-G_2\|.
\]
\end{lemma}

\begin{lemma}[\cite{s1977}, Theorem 3.4]\label{abl}
For $G_1,G_2\in \Cmn$ with $\rank(G_1) = \rank(G_2)$, we have 
$\|G_2^{\dagger}-G_1^{\dagger}\|\le(1+\sqrt{5})/2\|G_2^{\dagger}\|\|G_1^{\dagger}\|\|G_2-G_1\|$
if $\rank(G_1)<\min\{m,n\}$, and 
$\|G_2^{\dagger}-G_1^{\dagger}\|\leq\sqrt{2}\|G_2^{\dagger}\|\|G_1^{\dagger}\|\|G_2-G_1\|$
if $\rank(G_1) =\min\{m,n\}$.
\end{lemma}

On the relationship between  the generalized singular values of the $(m,p,n)$-GMP $\{A, L\}$ and the singular values of $A$, we have the following lemma.

\begin{lemma}[\cite{tikhonov}, Theorem 2.4]\label{sigma}
For the $(m,p,n)$-GMP $\{A, L\}$, we have 
$1/\|\ca^\dag\| \le \bar{\sigma}_{q-i+1}/\sigma_{i} \le \|\ca\|$ for all $\sigma_{i}\neq0$, where $\ca=[A^T,L^T]^T$ and $\bar{\sigma}_{1} \geq \bar{\sigma}_{2} \geq \cdots \geq \bar{\sigma}_{q} \geq 0$ denote the singular values of $A$ with $q=\min\{m,n\}$. 
\end{lemma}

\subsubsection{Overdetermined case}
On the error analysis of Algorithm  \ref{alg:alg4}, we have the following result.
\begin{theorem}\label{thm:errb1}
Let $\{A, L\}$ be an $(m,p,n)$-GMP with $m\ge n$. Assume that the approximate GSVD of the GMP $\{A, L\}$ is computed via Algorithm \ref{alg:alg4} with uniform random test matrices. Suppose $\bx_{\lambda}$ is the solution of \eqref{tikhonov} and $\bx_{\lambda}^{R}$ is the minimum 2-norm solution of the problem
$
\min_{\bx\in\mathbb{R}^{n}}\left\| [(PP^{T}A QQ^{T})^T, \lambda (LQQ^{T})^T]^T  \bx-[ \bb^T,  \mathbf{0}^T]^T\right\|.
$
Then we have
\[
\frac{\|\bx_{\lambda}^{R}-\bx_{\lambda}\|}{\|\bx_{\lambda}\|} 
\le \epsilon \Big( \frac{1+\sqrt{5}}{2} \xi^{2}\nu_{\lambda} + \sqrt{2} \xi\|\mathcal{A}^{\dagger}\|\nu_{\lambda} \Big) + \frac{1+\sqrt{5}}{2} \lambda  \gamma_1 \xi^{2}\nu_{\lambda},,
\]
where $\mathcal{A} = [A^{T},\lambda L^{T}]^{T}$, $\nu_{\lambda} = \| \bb \|  / \|\bx_{\lambda}\|$, $\xi$  is a constant depending on $\lambda$, $A$, $L$, $\Omega_P$, and $\gamma_{1}$ is a constant depending on $L$, $\Omega_P$, and $\Omega_Q$.
\end{theorem}

\begin{proof}
Let $\tilde{\mathcal{A}} = [(PP^{T}A)^{T},\lambda L^{T}]^{T}$ and
define $\delta\mathcal{A} = \tilde{\mathcal{A}}-\mathcal{A} = [(PP^{T}A-A)^T, 0]^T$. We note that $\rank([(P^TA)^T,L^T]^T)=n$ a.s. in Algorithm \ref{alg:alg4} since $\mathcal{A}$ is of full column rank. This shows that $\rank([(PP^TA)^T,\lambda L^T]^T)=n$ a.s. for  all $\lambda\neq 0$. Let $\bar{\bx}_{\lambda}^{R}$ be the minimum $2$-norm solution of the problem
$\min_{\bx\in\mathbb{R}^{n}}\|\tilde{\mathcal{A}}\bx-[\bb^T, \mathbf{0}^T]^T \|$.
Thus, 
\[
\|\bx_{\lambda}^{R}-\bx_{\lambda}\| \leq \|\bx_{\lambda}^{R}-\bar{\bx}_{\lambda}^{R}\| + \|\bar{\bx}_{\lambda}^{R} - \bx_{\lambda}\|,
\] 
where $\bx_{\lambda}=\mathcal{A}^\dag[\bb^T,{\bf 0}^T]^T$, $\bar{\bx}_{\lambda}^{R}=\tilde{\mathcal{A}}^\dag[\bb^T,{\bf 0}^T]^T$, and $\bx_{\lambda}^{R}=(\tilde{\mathcal{A}}QQ^T)^\dag [\bb^T,{\bf 0}^T]^T$.

We first estimate $\|\bx_{\lambda}^{R}-\bar{\bx}_{\lambda}^{R}\|$. By Lemma \ref{ab}  we obtain
\begin{eqnarray}\label{eq-1}
  &&\|\bx_{\lambda}^{R}-\bar{\bx}_{\lambda}^{R}\| \leq \|(\tilde{\mathcal{A}}QQ^{T})^{\dagger} - \tilde{\mathcal{A}}^{\dagger} \| \|\bb\| \nonumber\\
&\leq& \frac{1+\sqrt{5}}{2} \|
\tilde{\mathcal{A}}(QQ^{T}-I) \| \max\Big\{ \|(\tilde{\mathcal{A}}QQ^{T})^{\dagger}\|^{2}, \| \tilde{\mathcal{A}}^{\dagger} \|^{2}\Big\} \|\bb\| \nonumber\\
&\leq& \frac{1+\sqrt{5}}{2} \Big(\epsilon + \lambda \|L(QQ^{T}-I) \| \Big) \max\Big\{ \|(\tilde{\mathcal{A}}QQ^{T})^{\dagger}\|^{2}, \| \tilde{\mathcal{A}}^{\dagger} \|^{2}\Big\} \|\bb\|,
\end{eqnarray}
where the last inequality uses the fact that $\|PP^TA(QQ^T-I)\|_F\le \|P^TA- P^TAQQ^T\|_F$ $ \le\epsilon$. Using \cite[Corollary 4.4]{ss1990} we have 
\begin{eqnarray}\label{A1qqinv}
\|(\tilde{\mathcal{A}}QQ^{T})^{\dagger}\| = \|Q(\tilde{\mathcal{A}}Q)^{\dagger}\| = \|(\tilde{\mathcal{A}}Q)^{\dagger}\| \leq \|\tilde{\mathcal{A}}^{\dagger}\|.
\end{eqnarray}
Substituting  \eqref{A1qqinv} into   \eqref{eq-1}, we obtain
\begin{eqnarray}\label{eq-11}
\|\bx_{\lambda}^{R}-\bar{\bx}_{\lambda}^{R}\| \leq \frac{1+\sqrt{5}}{2} (\epsilon + \lambda \|L(QQ^{T}-I)\|) \|\tilde{\mathcal{A}}^{\dagger}\|^{2} \| \bb \| .
\end{eqnarray}

Next, we estimate $\|\bar{\bx}_{\lambda}^{R}-\bx_{\lambda}\|$. We note that $\rank(\tilde{\mathcal{A}})=\rank(\mathcal{A})=n$ a.s. for all $\lambda \neq 0$. We have by Lemma \ref{abl} and using $\|PP^{T}A-A\|<\epsilon$,
\begin{eqnarray}\label{eq-2}
\|\bar{\bx}_{\lambda}^{R}-\bx_{\lambda}\| \le \|\tilde{\mathcal{A}}^{\dagger} - \mathcal{A}^{\dagger} \| \| \bb \|  
\leq \sqrt{2} \|\tilde{\mathcal{A}} - \mathcal{A} \|  \|\tilde{\mathcal{A}}^{\dagger}\| \| \mathcal{A}^{\dagger} \| \| \bb \|  \le \sqrt{2} \epsilon \|\tilde{\mathcal{A}}^{\dagger}\| \|\mathcal{A}^{\dagger}\| \| \bb \| .
\end{eqnarray}

Using \eqref{eq-11}--\eqref{eq-2}, we obtain the error bound
\begin{eqnarray*}\label{eq}
&&\frac{\|\bx_{\lambda}^{R}-\bx_{\lambda}\|}{\|\bx_{\lambda}\|} 
\leq \frac{1+\sqrt{5}}{2} (\epsilon + \lambda \|L(QQ^{T}-I)\|) \|\tilde{\mathcal{A}}^{\dagger}\|^{2}\nu_{\lambda} + \sqrt{2} \epsilon \|\tilde{\mathcal{A}}^{\dagger}\| \|\mathcal{A}^{\dagger}\|\nu_{\lambda} \\
&&\le \epsilon \Big( \frac{1+\sqrt{5}}{2} \xi^{2}\nu_{\lambda} + \sqrt{2} \xi\|\mathcal{A}^{\dagger}\|\nu_{\lambda} \Big) + \frac{1+\sqrt{5}}{2} \lambda  \gamma_1 \xi^{2}\nu_{\lambda},
\end{eqnarray*}
where $\nu_{\lambda} = \| \bb \|  / \|\bx_{\lambda}\|$, $\xi=\|\tilde{\mathcal{A}}^{\dagger}\|$ depending on $\lambda$, $A$, $L$, $\Omega_P$,  and $\gamma_{1} = \|L(QQ^{T}-I)\|$ depending on $L$, $\Omega_P$, and $\Omega_Q$.
\end{proof}

\begin{remark}
In Theorem {\rm \ref{thm:errb1}}, let $\bar{\sigma}_{1}\geq\bar{\sigma}_{2}\geq\cdots\geq\bar{\sigma}_{n}\ge 0$ be the singular values of $A$. If the tolerance  $\epsilon$ and the regularization parameter  $\lambda$ is such that $\epsilon = \mathcal{O}(\bar{\sigma}_{k+1})=\lambda$ for some target rank $k$, then
\[
\frac{\|\bx_{\lambda}^{R}-\bx_{\lambda}\|}{\|\bx_{\lambda}\|} 
\le \Big( \frac{1+\sqrt{5}}{2}\bar{c}_1 \xi^{2}\nu_{\lambda} + \sqrt{2} \bar{c}_1\xi\|\mathcal{A}^{\dagger}\|\nu_{\lambda}  + \frac{1+\sqrt{5}}{2} \bar{c}_2\gamma_1 \xi^{2}\nu_{\lambda} \Big)\bar{\sigma}_{k+1}+\mathcal{O}(\bar{\sigma}_{k+1}^{2}),
\]
where $\bar{c}_1$ and $\bar{c}_2$ are two constants depending on $\epsilon$ and $\lambda$, respectively.
\end{remark}

\subsubsection{Underdetermined case}
For Algorithm \ref{alg:alg5}, we have the following error bound.
\begin{theorem}
Let $\{A, L\}$ be an $(m,p,n)$-GMP with $m\le n$. Let $\bar{\sigma}_{1}\geq\bar{\sigma}_{2}\geq\cdots\geq\bar{\sigma}_{m}\ge 0$ be the singular values of $A$. Assume that the approximate GSVD of the GMP $\{A, L\}$ is computed via Algorithm {\rm\ref{alg:alg5}} with uniform random test matrices. Suppose $\bx_{\lambda}$ is the solution of \eqref{tikhonov} and $\bx_{\lambda}^{R}$ is the minimum 2-norm solution of the problem
\[\min_{\bx\in\mathbb{R}^{n}}\left\|\left[\begin{array}{c}
PP^{T}A QQ^{T} \\
                                          \lambda LQQ^{T} 
                                        \end{array}
\right] \bx-\left[\begin{array}{c}
                 \bb \\
                 \mathbf{0} 
               \end{array}
\right]\right\|.
\]
If the tolerance  $\epsilon$ and the regularization parameter  $\lambda$ is such that $\epsilon = \mathcal{O}(\bar{\sigma}_{k+1})=\lambda$ for some target rank $k$, then
\begin{eqnarray*}
\frac{\|\bx_{\lambda}^{R}-\bx_{\lambda}\|}{\|\bx_{\lambda}\|}
\le  \|\ca^\dag\| \Big( c\bar{c}_1 (\xi\tilde{\nu}_{\lambda} +\nu_\lambda ) 
+c\bar{c}_2\gamma_1 \nu_\lambda 
+ \Big(1+c \kappa(\ca) \nu_\lambda \Big)\gamma_2 \Big)\bar{\sigma}_{k+1}
+ \mathcal{O}(\bar{\sigma}_{k+1}^{2}),
\end{eqnarray*}
where $\mathcal{A} = [A^{T},\lambda L^{T}]^{T}$, $c= (1+\sqrt{5})/2$, $\nu_{\lambda} = \|\mathcal{A}^{\dagger}\| \tilde{\nu}_{\lambda}$ with $\tilde{\nu}_{\lambda} =\| \bb \|  / \|x_{\lambda}\|$, $\kappa(\mathcal{A}) = \|\mathcal{A}\|\|\mathcal{A}^{\dagger}\|$, $ \gamma_{1}$ is a constant depending on $L$,  $\Omega_Q$, $\bar{c}_1$ is a constant depending on $\epsilon$, $\bar{c}_2$ is a constant depending on  $\lambda$, and $\xi$, $\gamma_{2}$ are two constants depending on $\lambda$,  $\Omega_Q$, $\Omega_P$, $A$, $L$. 
\end{theorem}

\begin{proof}
Assume that $\underline{\bx}_{\lambda}^{R}$ is the minimum 2-norm solution of the problem
$\min_{\bx\in\mathbb{R}^{n}}$ $\|\mathcal{A}QQ^{T}\bx-[\bb^T, \mathbf{0}^T]^T \|$.
Thus,  
\[
\|\bx_{\lambda}^{R}-\bx_{\lambda}\| \leq \|\bx_{\lambda}^{R}-\underline{\bx}_{\lambda}^{R}\| + \|\underline{\bx}_{\lambda}^{R} - \bx_{\lambda}\|, 
\]
where $\bx_{\lambda}=\ca^\dag[\bb^T,{\bf 0}^T]^T$, $\underline{\bx}_{\lambda}^{R}=(\ca QQ^T)^\dag[\bb^T,{\bf 0}^T]^T$, and $\bx_{\lambda}^{R}=(\tilde{\mathcal{A}}QQ^T)^\dag [\bb^T,{\bf 0}^T]^T$ with $\tilde{\mathcal{A}}=[(PP^{T}A)^{T},\lambda L^{T}]^{T}$.

We first estimate $\|\bx_{\lambda}^{R}-\underline{\bx}_{\lambda}^{R}\|$.  Since $\rank(\tilde{\mathcal{A}}QQ^T)=\rank(\ca Q)=\rank(\mathcal{A}QQ^T)$ $=\ell_1\le m\le n$ a.s.  for all $\lambda \neq 0$, by Lemma \ref{abl}, we obtain
\begin{eqnarray}\label{i}
  \| \bx_{\lambda}^{R}-\underline{\bx}_{\lambda}^{R}\| &\leq& \|(\tilde{\mathcal{A}}QQ^{T})^{\dagger} - (\mathcal{A}QQ^{T})^{\dagger}\| \| \bb \|  \nonumber\\
  &\leq& \frac{1+\sqrt{5}}{2} \|\tilde{\mathcal{A}}QQ^{T} - \mathcal{A}QQ^{T}\| \|(\tilde{\mathcal{A}}QQ^{T})^{\dagger}\| \|(\mathcal{A}QQ^{T})^{\dagger}\| \| \bb \|  \nonumber\\
  &\leq& \frac{1+\sqrt{5}}{2} \epsilon \|\tilde{\mathcal{A}}^{\dagger}\| \|\mathcal{A}^{\dagger}\| \| \bb \| ,
\end{eqnarray}
where the third inequality holds due to the fact that
$\|(\mathcal{A}QQ^{T})^{\dagger}\| = \|Q(\mathcal{A}Q)^{\dagger}\| = \|(\mathcal{A}Q)^{\dagger}\| \leq \|\mathcal{A}^{\dagger}\|$,  $\|(\tilde{\ca}QQ^{T})^{\dagger}\| = \|Q(\tilde{\ca}Q)^{\dagger}\| = \|(\tilde{\ca}Q)^{\dagger}\| \leq \|\tilde{\ca}^{\dagger}\|$, and $\|\tilde{\mathcal{A}}QQ^{T} - \mathcal{A}QQ^{T}\| = \|PP^{T}AQ-AQ\| < \epsilon$.

We now estimate $\|\underline{\bx}_{\lambda}^{R} - \bx_{\lambda}\|$. Let  the GSVD of the $(m,p,n)$-GMP $\{A, L\}$  be given by  \eqref{def} with $B=L$. Denote 
\BE \label{def:al}
A_{\ell_1} = U\left[
             \begin{array}{cc}
               0 &  \\
                 & D_1^{(n-\ell_1)} \\
             \end{array}
           \right]X^{-1},\quad 
L_{\ell_1} =V\left[
             \begin{array}{cc}
             D_2^{(n-\ell_1)}  &  \\
                 & 0 \\
             \end{array}
           \right]X^{-1},
\EE
where $D_1^{(n-\ell_1)}$ equals $D_1$ with the $(r+s-\ell_1)$ smallest $\alpha_{i}$'s being replaced by zeros and  $D_2^{(n-\ell_1)}$ equals $D_2$ with the $(r+s-\ell_1)$  largest $\beta_{i}$'s being replaced by ones.  Assume that $\bx_{\ell_1,\lambda}^{L}$ is the minimum $2$-norm solution of the problem
$\min_{\bx\in\mathbb{R}^{n}} \|[A_{\ell_1}^T, \lambda L_{\ell_1}^T]^T \bx-[\bb^T, \mathbf{0}^T]^T\|$.
Then $\bx_{\ell_1,\lambda}^{L} = \mathcal{A}_{\ell_1}^{\dagger}[\bb^{T},\mathbf{0}^{T}]^{T}$ with
 $\mathcal{A}_{\ell_1} = [A_{\ell_1}^{T},\lambda L_{\ell_1}^{T}]^{T}$. 
By the definition of $\ca_{\ell_1}$ we have
\[\mathcal{A}_{\ell_1} = \mathcal{A}X\left[
                                  \begin{array}{cc}
                                   0 &   \\
                                      & I_{\ell_1} \\
                                  \end{array}
                                \right]X^{-1} \quad\mbox{and}\quad
\mathcal{A}_{\ell_1}^{\dagger} = (I-X_{1}X_{1}^{\dagger})\mathcal{A}^{\dagger},
\]
where $X_{1} = X(:,1:n-\ell_1)$.  Then, 
$
\|\underline{\bx}_{\lambda}^{R}-\bx_{\lambda}\| \leq \|\underline{\bx}_{\lambda}^{R}-\bx_{\ell_1,\lambda}^{L}\| + \|\bx_{\ell_1,\lambda}^{L} - \bx_{\lambda}\|.
$

In the following, we estimate $\|\underline{\bx}_{\lambda}^{R}-\bx_{\ell_1,\lambda}^{L}\|$. We note that 
$\rank(\mathcal{A}QQ^{T})=\ell_1=\rank(\mathcal{A}_{\ell_1})$ a.s. for all $\lambda\ne 0$.
By Lemma \ref{abl} we obtain
\begin{eqnarray}\label{iia1}
 \|\underline{\bx}_{\lambda}^{R}-\bx_{\ell_1,\lambda}^{L}\| &\leq& \|(\mathcal{A}QQ^{T})^{\dagger} - \mathcal{A}_{\ell_1}^{\dagger}\|\| \bb \|  \nonumber\\
 &\leq& \frac{1+\sqrt{5}}{2} \|\mathcal{A}QQ^{T} - \mathcal{A}_{\ell_1}\| \|(\mathcal{A}QQ^{T})^{\dagger}\| \|\mathcal{A}_{\ell_1}^{\dagger}\| \| \bb \| .
\end{eqnarray}
Using $\|\mathcal{A}QQ^{T} - \mathcal{A}_{\ell_1}\| \leq \|\mathcal{A}QQ^{T} - \mathcal{A}\| + \|\mathcal{A} - \mathcal{A}_{\ell_1}\|$ and
\[\|\mathcal{A}QQ^{T} - \mathcal{A}\| \leq \|AQQ^{T}-A\|+\lambda\|L(QQ^{T}-I)\| \leq \epsilon + \lambda\|L(QQ^{T}-I)\|.\]
From the GSVD of $\{A,L\}$, we have
\[\|\mathcal{A} - \mathcal{A}_{\ell_1}\| = \left\|\mathcal{A}X\left[
                                                           \begin{array}{cc}
                                                             I_{n-\ell_1} &   \\
                                                               & 0 \\
                                                           \end{array}
                                                         \right]
X^{-1}\right\| = \|\mathcal{A}[X_{1},0]X^{-1}\| \leq \|X_{1}\|\|X^{-1}\|\|\mathcal{A}\|. \]
Furthermore,  $\|\mathcal{A}_{\ell_1}^{\dagger}\| = \|(I-X_{1}X_{1}^{\dagger})\mathcal{A}^{\dagger}\| \leq \|\mathcal{A}^{\dagger}\|$.
Hence, \eqref{iia1} is reduced to
\begin{eqnarray}\label{iia}
 \|\underline{\bx}_{\lambda}^{R}-\bx_{\ell,\lambda}^{L}\| \leq \frac{1+\sqrt{5}}{2} \Big(\epsilon + \lambda\|L(QQ^{T}-I)\| + \|X_{1}\|\|X^{-1}\|\|\mathcal{A}\| \Big) \|\mathcal{A}^{\dagger}\|^{2} \| \bb \| . 
\end{eqnarray}

Next, we estimate $\|\bx_{\ell_1,\lambda}^{L} - \bx_{\lambda}\|$. We note that $\bx_{\lambda} - \bx_{\ell_1,\lambda}^{L} = X_{1}X_{1}^{\dagger}\bx_{\lambda}$. Thus,
\begin{eqnarray}\label{iib}
\|\bx_{\lambda} - \bx_{\ell_1,\lambda}^{L}\| \leq \|X_{1}\|\|X_{1}^{\dagger}\|\|\bx_{\lambda}\| \leq \|X_{1}\|\|X^{-1}\|\|\bx_{\lambda}\|.
\end{eqnarray}
From the GSVD of $\{A,L\}$, and since $\mathcal{A}$ is of full column rank, we have
\[
\ca X_{1} =  \left[
                     \begin{array}{cc}
                       U &  \\
                        & V \\
                     \end{array}
                   \right]  \left[
                     \begin{array}{c}
                     \Sigma_A^{(1)}\\
                     \lambda\Sigma_L^{(1)}
                     \end{array}
                   \right] \quad\mbox{and}\quad
X_{1} =\ca^\dag \left[
                     \begin{array}{cc}
                       U &  \\
                        & V \\
                     \end{array}
                   \right]  \left[
                     \begin{array}{c}
                     \Sigma_A^{(1)}\\
                     \lambda\Sigma_L^{(1)}
                     \end{array}
                   \right],                  
\]
where $\Sigma_A^{(1)} = \Sigma_A(:,1:n-\ell_1)$ and $\Sigma_L^{(1)} = \Sigma_L(:,1:n-\ell_1)$. It is easy to check that
\BE\label{X1}
\|X_{1}\| \le \|\mathcal{A}^{\dagger}\|\sqrt{\sigma_{m-\ell_1}^{2}+\lambda^{2}} \leq \|\mathcal{A}^{\dagger}\|\sqrt{\|X\|^{2}\bar{\sigma}_{\ell_1+1}^{2}+\lambda^{2}},
\EE
where the second inequality uses the fact that $\sigma_{m-\ell_1} \leq \|X\|\bar{\sigma}_{\ell_1+1}$ from Lemma \ref{sigma}.

From \eqref{i}, \eqref{iia}, \eqref{iib}, \eqref{X1}, $\epsilon =\bar{c}_1 \bar{\sigma}_{k+1} +\mathcal{O}(\bar{\sigma}_{k+1}^{2})$ and $\lambda =\bar{c}_2\bar{\sigma}_{k+1} + \mathcal{O}(\bar{\sigma}_{k+1}^{2})$, we obtain the error bound
\begin{eqnarray*}
\frac{\|\bx_{\lambda}^{R}-\bx_{\lambda}\|}{\|\bx_{\lambda}\|}
\le  \|\ca^\dag\| \Big( c\bar{c}_1 (\xi\tilde{\nu}_{\lambda} +\nu_\lambda ) 
+c\bar{c}_2\gamma_1 \nu_\lambda 
+ \Big(1+c \kappa(\ca) \nu_\lambda \Big)\gamma_2 \Big)\bar{\sigma}_{k+1}
+ \mathcal{O}(\bar{\sigma}_{k+1}^{2}),
\end{eqnarray*}
where  $\nu_{\lambda} = \|\mathcal{A}^{\dagger}\| \tilde{\nu}_{\lambda}$ with $\tilde{\nu}_{\lambda} =\| \bb \|  / \|x_{\lambda}\|$,  $\xi=\|\tilde{\mathcal{A}}^{\dagger}\|$, $\gamma_{1} = \|L(QQ^{T}-I)\|$, $\gamma_{2} = \|X^{-1}\|\sqrt{\|X\|^{2}+\bar{c}_2^2}$ depending on $\lambda$, $\Omega_Q$, $\Omega_P$, $A$, $L$.
\end{proof}

\section{Numerical examples}\label{sec6}
In this section, we present some numerical experiments  of Algorithms \ref{alg:alg4} and \ref{alg:alg5} for solving the  large-scale Tikhonov regularization problem \eqref{tikhonov}. For demonstration purposes, we compare the proposed algorithms with  the classical GSVD, TGSVD in \cite{tikhonov}, RGSVD in \cite{2016,xiang2015} and MTRSVD in \cite{jia2018}. All numerical tests were carried out in MATLAB R2021b running on a laptop with AMD Ryzen 9 5900HX of  3.30 GHz CPU. 

In our numerical experiments, we adopt the following notation:
\begin{itemize}
\item $l_s$ means the number of samples. It is determined by the number of columns of the orthonormal matrix $Q$ of RGSVD  \cite{2016}. It is also used as the truncation parameter of the TGSVD, RGSVD  \cite{xiang2015} and MTRSVD. Denote by $l_{1}$ and $l_{2}$  the column numbers of  $P$ and $Q$ generated by Algorithms \ref{alg:alg4} and \ref{alg:alg5}.

\item $\hat{\bx}$ is the true solution of the ill-posed least squares problem $\min_{\bx\in \Rn} \|A\bx-\bb\|^2$, $\bx_{\rm gsvd}$ is the approximate solution to (\ref{tikhonov}) by GSVD (which is computed by {\tt gsvd} in {\tt MATLAB}),  $\bx_{\rm tgsvd}$ is the  TGSVD solution to \eqref{tikhonov},
$\bx_{\rm rgsvd}$ is the approximate regularized solution via RGSVD, $\bx_{\rm mtrsvd}$ is the  MTRSVD solution to \eqref{tikhonov}, $\bx_{\rm alg. \ref{alg:alg4}}$ and  $\bx_{\rm alg. \ref{alg:alg5}}$ is the approximate regularized solutions via  Algorithms \ref{alg:alg4} and \ref{alg:alg5}, respectively.

\item $E_{\rm gsvd}=\|\hat{\bx}-\bx_{\rm gsvd}\|/\|\hat{\bx}\|$, $E_{\rm tgsvd}=\|\hat{\bx}-\bx_{\rm tgsvd}\|/\|\hat{\bx}\|$, $E_{\rm rgsvd}=\|\hat{\bx}-\bx_{\rm rgsvd}\|/\|\hat{\bx}\|$,  $E_{\rm mtrsvd}=\|\hat{\bx}-\bx_{\rm mtrsvd}\|/\|\hat{\bx}\|$,  $E_{\rm alg. \ref{alg:alg4}}=\|\hat{\bx}-\bx_{\rm alg. \ref{alg:alg4}}\|/\|\hat{\bx}\|$,  and $E_{\rm alg. \ref{alg:alg5}}=\|\hat{\bx}-\bx_{\rm alg. \ref{alg:alg5}}\|/\|\hat{\bx}\|$ represent the relative errors. 

\item $t_{\rm gsvd}$, $t_{\rm tgsvd}$, $t_{\rm rgsvd}$, $t_{\rm mtrsvd}$, $t_{\rm alg. \ref{alg:alg4}}$, and $t_{\rm alg. \ref{alg:alg5}}$ denote the running time  (in seconds) of GSVD, TGSVD, RGSVD and MTRSVD,  Algorithm \ref{alg:alg4}, and Algorithm \ref{alg:alg5},  respectively.

\item The tolerance $\epsilon$ is set to be $\epsilon=10^{-2}$.
\end{itemize}
\begin{example} \label{ex1} {\rm (overdetermined examples from Hansen Tools \cite{Hansen})}
The test examples  {\tt baart}, {\tt deriv2}, {\tt foxgood}, {\tt gravity}, {\tt heat}, {\tt phillips} and {\tt shaw} are from Hansen Tools \cite{Hansen}. Let
$
\bb=\bb+\delta \|\bb\| \bd\zeta/\|\bd\zeta\|,
$
where $\delta>0$ is the relative noise level, $\bd\zeta={\tt randn}(m,1)$ is a random vector.  In our numerical tests, we set $\delta=10^{-3}$,  $m=n=2048$, the parameter $\lambda$ is selected by the GCV method or the L-curve rule,  and the regularization matrix is $L=(l_{ij})\in\R^{(n-1)\times n}$ with $l_{ii}=1$ and $l_{i,i+1}=-1$ for all $1\le i\le n-1$ and $l_{ij}=0$ otherwise \cite{tikhonov}. 
\end{example}

\begin{table}[!ht]
 \centering {\scriptsize
\caption{The comparison results (GCV) for Example \ref{ex1} with $m=n=2048$}
\label{Table:image_over_gcv}
\begin{tabular}{l|ccccccc}
  \hline
  $n=2048$ & {\tt baart}  & {\tt deriv2} & {\tt foxgood} & {\tt gravity} & {\tt heat}  & {\tt phillips}  & {\tt shaw}   \\
  \hline
$E_{\rm gsvd}$  &	1.07E-01	&	1.56E-02	&	2.96E-02	&	1.92E-02	&	2.11E-02	&	7.39E-03	&	3.30E-02	\\
$E_{\rm tgsvd}$ in \cite{tikhonov} &	1.03E-01	&	8.93E-02	&	5.61E-02	&	2.47E-02	&	4.60E-02	&	7.84E-02	&	3.75E-02	\\
$E_{\rm rgsvd}$ in \cite{2016}  &	1.08E-01	&	9.86E-02	&	5.52E-02	&	1.76E-02	&	5.12E-02	&	6.91E-03	&	3.94E-02	\\
$E_{\rm rgsvd}$ in \cite{xiang2015}  &	1.17E-01	&	3.38E-02	&	1.74E-02	&	1.34E-02	&	5.75E-02	&	1.65E-02	&	4.61E-02	\\
$E_{\rm mtrsvd}$ in \cite{jia2018}  &	1.03E-01	&	8.90E-02	&	5.60E-02	&	2.57E-02	&	4.84E-02	&	5.02E-02	&	3.91E-02	\\
$E_{\rm alg. \;\ref{alg:alg4} }$ &	1.17E-01	&	4.36E-02	&	1.45E-02	&	1.07E-02	&	4.59E-02	&	6.90E-03	&	4.43E-02	\\
$t_{\rm gsvd}$  &	2.8468 	&	3.1576 	&	2.8428 	&	2.6728 	&	2.7477 	&	2.8402 	&	2.7061 	\\
$t_{\rm tgsvd}$ in \cite{tikhonov} &	2.7326 	&	2.9571 	&	2.6380 	&	2.6686 	&	2.6651 	&	3.0562 	&	2.6721 	\\
$t_{\rm rgsvd}$ in \cite{2016}  &	2.0730 	&	2.0586 	&	2.0837 	&	2.1551 	&	2.0423 	&	2.3683 	&	2.2010 	\\
$t_{\rm rgsvd}$ in \cite{xiang2015}  &	0.2610 	&	0.2528 	&	0.2053 	&	0.2817 	&	0.4768 	&	0.5121 	&	0.2637 	\\
$t_{\rm mtrsvd}$ in \cite{jia2018}  &	0.3444 	&	1.7598 	&	1.9283 	&	0.9727 	&	0.7929 	&	0.6043 	&	1.0306 	\\
$t_{\rm alg.\; \ref{alg:alg4} }$ &	0.0637 	&	0.0451 	&	0.0609 	&	0.0508 	&	0.0612 	&	0.0622 	&	0.0475 	\\
$l_{s}$&	4	&	6	&	3	&	11	&	23	&	30	&	8	\\
$l_{1}\times l_{2}$&	$4\times4$	&	$4\times4$	&	$3\times3$	&	$11\times11$	&	$25\times25$	&	$30\times30$	&	$8\times8$	\\
$\lambda_{\rm gsvd}$  &	3.23E-01	&	1.35E-02	&	1.08E-01	&	5.22E-01	&	9.74E-03	&	1.40E+00	&	1.92E-01	\\
$\lambda_{\rm tgsvd}$ in \cite{tikhonov} &	1.24E-02	&	1.47E-01	&	3.06E-02	&	2.24E-01	&	4.78E-02	&	6.10E-01	&	1.87E-02	\\
$\lambda_{\rm rgsvd}$ in \cite{2016}  &	1.05E+00	&	3.34E-01	&	2.02E+00	&	9.96E-01	&	4.95E-02	&	1.65E+00	&	3.96E-01	\\
$\lambda_{\rm rgsvd}$ in \cite{xiang2015}  &	1.02E+00	&	1.83E-01	&	1.54E+00	&	8.18E-01	&	3.40E-02	&	4.12E-01	&	3.18E-01	\\
$\lambda_{\rm alg. \;\ref{alg:alg4} }$ &	1.02E+00	&	9.83E-01	&	1.66E+00	&	6.23E-01	&	4.41E-02	&	1.70E+00	&	2.62E-01	\\
    \hline
\end{tabular}}
\end{table}

\begin{table}[!ht]
 \centering {\scriptsize
\caption{The comparison results (L-curve) for  Example \ref{ex1} with $m=n=2048$}
\label{Table:image_over_lcurve}
\begin{tabular}{l|ccccccc}
  \hline
  $n=2048$ & {\tt baart}  & {\tt deriv2} & {\tt foxgood} & {\tt gravity} & {\tt heat}  & {\tt phillips}  & {\tt shaw}   \\
  \hline
$E_{\rm gsvd}$  &	1.10E-01	&	2.78E-02	&	5.35E-02	&	2.02E-02	&	2.05E-02	&	7.47E-03	&	3.93E-02	\\
$E_{\rm tgsvd}$ in \cite{tikhonov} &	1.03E-01	&	8.93E-02	&	5.61E-02	&	8.78E-02	&	9.61E-02	&	1.13E-01	&	3.75E-02	\\
$E_{\rm rgsvd}$ in \cite{2016}  &	1.10E-01	&	9.54E-02	&	5.51E-02	&	2.02E-02	&	5.08E-02	&	7.19E-03	&	3.98E-02	\\
$E_{\rm rgsvd}$ in \cite{xiang2015}  &	1.24E-01	&	5.12E-02	&	3.38E-02	&	4.27E-02	&	7.43E-02	&	1.02E-01	&	3.61E-02	\\
$E_{\rm mtrsvd}$ in \cite{jia2018}  &	1.03E-01	&	8.51E-02	&	5.60E-02	&	2.40E-02	&	3.51E-02	&	8.68E-02	&	2.37E-02	\\
$E_{\rm alg. \;\ref{alg:alg4} }$ &	1.17E-01	&	5.10E-02	&	1.37E-02	&	2.66E-02	&	5.71E-02	&	2.14E-02	&	3.34E-02	\\
$t_{\rm gsvd}$  &	2.8422 	&	3.0393 	&	2.7792 	&	2.7421 	&	2.6902 	&	2.9127 	&	2.7544 	\\
$t_{\rm tgsvd}$ in \cite{tikhonov} &	2.8044 	&	2.8574 	&	2.6622 	&	2.7328 	&	2.6981 	&	3.0323 	&	2.6721 	\\
$t_{\rm rgsvd}$ in \cite{2016}  &	2.0990 	&	2.1359 	&	2.0535 	&	2.0732 	&	2.0937 	&	2.2601 	&	2.3128 	\\
$t_{\rm rgsvd}$ in \cite{xiang2015}  &	0.2657 	&	0.2580 	&	0.2297 	&	0.3286 	&	0.4807 	&	0.5748 	&	0.2635 	\\
$t_{\rm mtrsvd}$ in \cite{jia2018}  &	0.3444 	&	1.6572 	&	1.9283 	&	0.8469 	&	0.7465 	&	0.5847 	&	1.0637 	\\
$t_{\rm alg.\; \ref{alg:alg4} }$ &	0.0549 	&	0.0747 	&	0.0620 	&	0.0637 	&	0.0699 	&	0.1011 	&	0.0557 	\\
$l_{s}$&	4	&	6	&	3	&	14	&	27	&	32	&	8	\\
$l_{1}\times l_{2}$&	$4\times4$	&	$4\times4$	&	$3\times3$	&	$13\times13$	&	$24\times22$	&	$27\times19$	&	$9\times9$	\\
$\lambda_{\rm gsvd}$  &	1.27E+00	&	5.27E-02	&	8.25E-01	&	2.68E+00	&	1.44E-02	&	1.37E+00	&	4.29E-01	\\
$\lambda_{\rm tgsvd}$ in \cite{tikhonov} &	5.75E-02	&	1.47E-01	&	3.06E-02	&	9.89E-01	&	6.34E+00	&	6.19E-01	&	3.92E-02	\\
$\lambda_{\rm rgsvd}$ in \cite{2016}  &	1.30E+00	&	3.32E-01	&	2.95E+00	&	2.69E+00	&	5.92E-02	&	1.35E+00	&	4.45E-01	\\
$\lambda_{\rm rgsvd}$ in \cite{xiang2015}  &	9.96E-01	&	1.40E-01	&	1.15E+00	&	8.49E-02	&	1.50E-02	&	2.01E-02	&	2.86E-01	\\
$\lambda_{\rm alg. \;\ref{alg:alg4} }$ &	1.02E+00	&	1.01E+00	&	1.72E+00	&	1.76E-01	&	5.10E-02	&	2.45E-01	&	8.83E-02	\\
    \hline
\end{tabular}}
\end{table}

\begin{figure}[!ht]
	\centering
	\subfigure[{\tt baart}]
    {\includegraphics[width=0.4\linewidth,height=0.14\textheight]{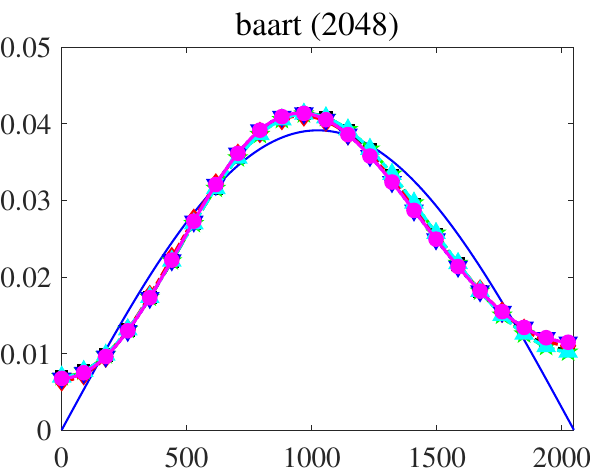}
	}	
	\subfigure[{\tt deriv2}]
    {\includegraphics[width=0.4\linewidth,height=0.14\textheight]{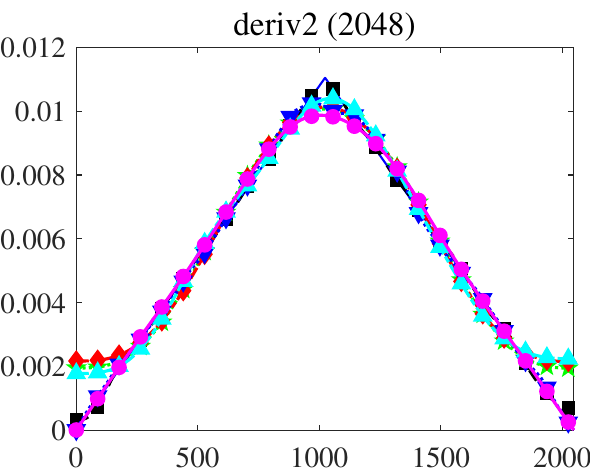}
	}
    \subfigure[{\tt foxgood}]
    {\includegraphics[width=0.4\linewidth,height=0.14\textheight]{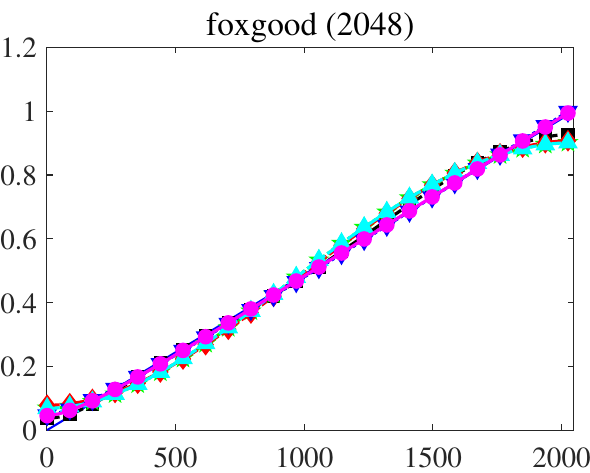}
	}
    \subfigure[{\tt gravity}]
    {\includegraphics[width=0.4\linewidth,height=0.14\textheight]{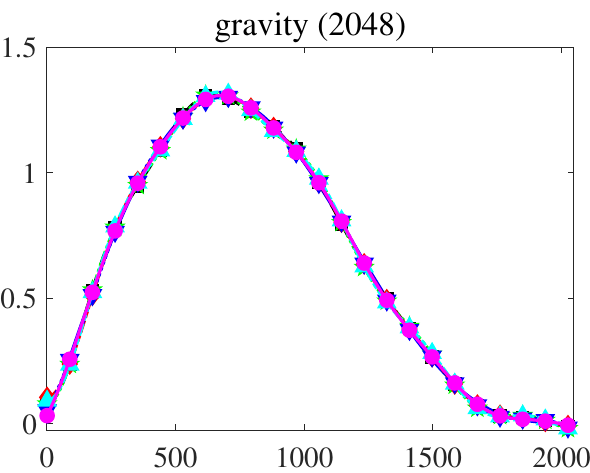}
	}
	\subfigure[{\tt heat}]
    {\includegraphics[width=0.4\linewidth,height=0.14\textheight]{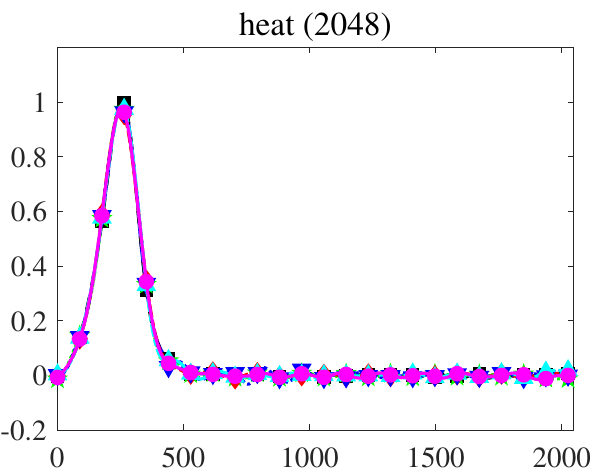}
	}	
	\subfigure[{\tt phillips}]
    {\includegraphics[width=0.4\linewidth,height=0.14\textheight]{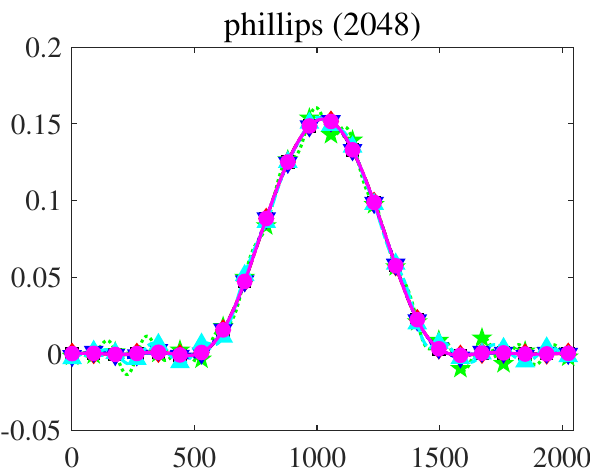}
	}
    \subfigure[{\tt shaw}]
    {\includegraphics[width=0.6\linewidth,height=0.14\textheight]{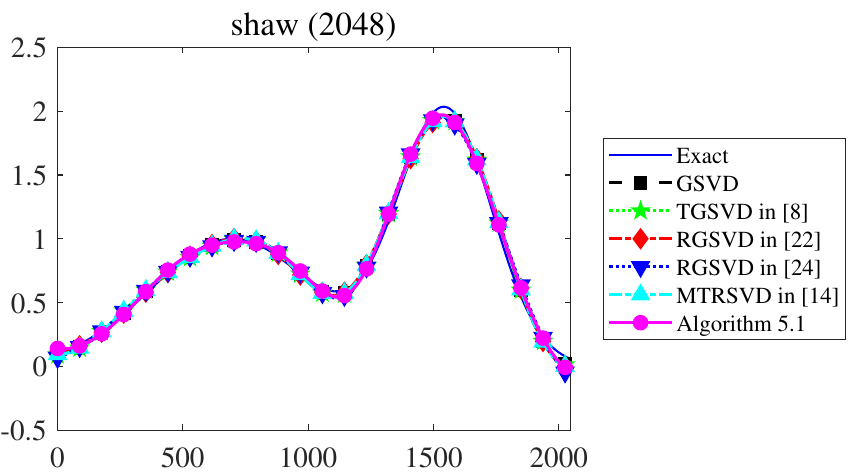}
	}
	\caption{Algorithm \ref{alg:alg4}  (GCV) for  Example \ref{ex1} with $m=n=2048$}
	\label{Fig:image_over_gcv}	
\end{figure}

\begin{figure}[!ht]
	\centering
	\subfigure[{\tt baart}]
    {\includegraphics[width=0.4\linewidth,height=0.14\textheight]{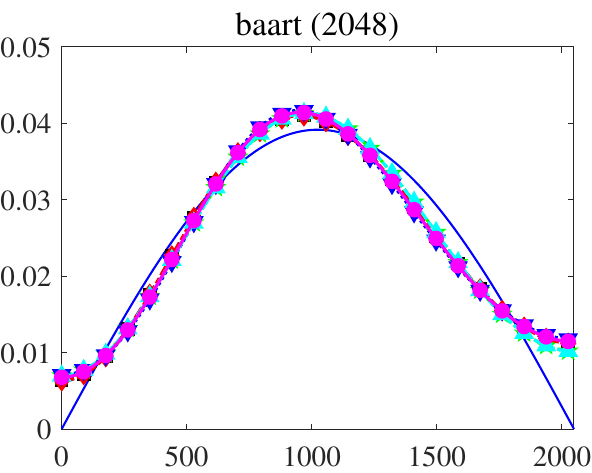}
	}
    \subfigure[{\tt deriv2}]
    {\includegraphics[width=0.4\linewidth,height=0.14\textheight]{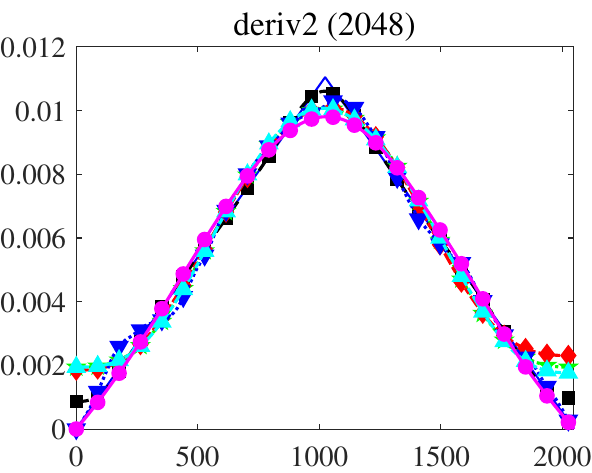}
	}
	\subfigure[{\tt foxgood}]
    {\includegraphics[width=0.4\linewidth,height=0.14\textheight]{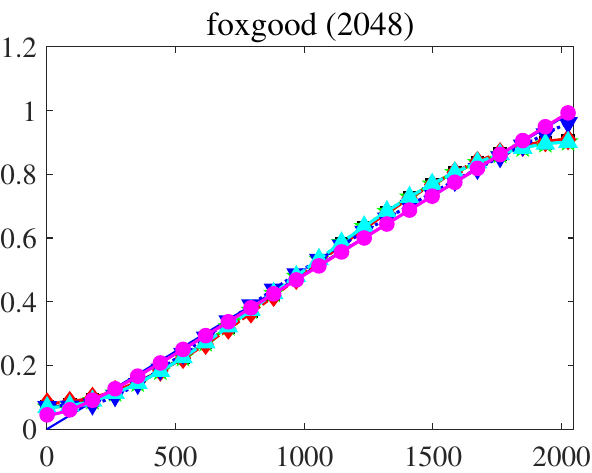}
	}
    \subfigure[{\tt gravity}]
    {\includegraphics[width=0.4\linewidth,height=0.14\textheight]{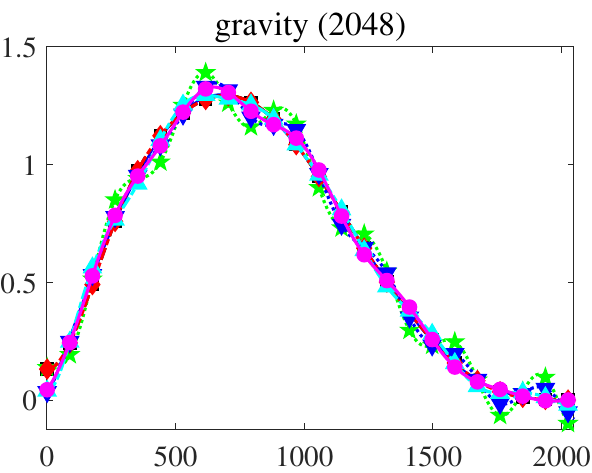}
	}
	\subfigure[{\tt heat}]
    {\includegraphics[width=0.4\linewidth,height=0.14\textheight]{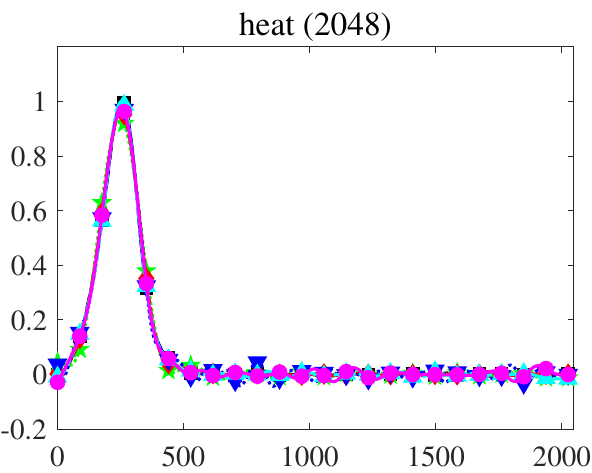}
	}	
	\subfigure[{\tt phillips}]
    {\includegraphics[width=0.4\linewidth,height=0.14\textheight]{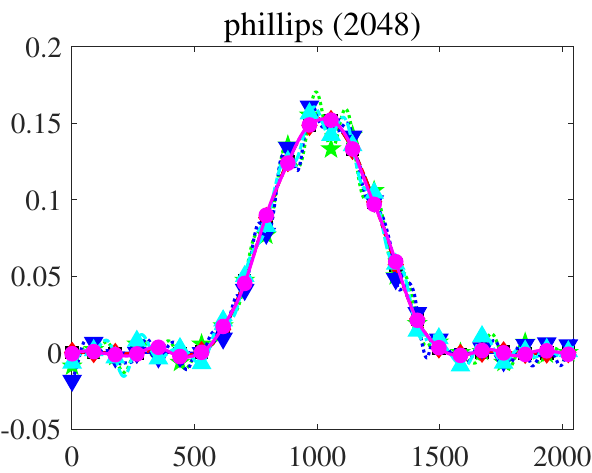}
	}
    \subfigure[{\tt shaw}]
    {\includegraphics[width=0.6\linewidth,height=0.14\textheight]{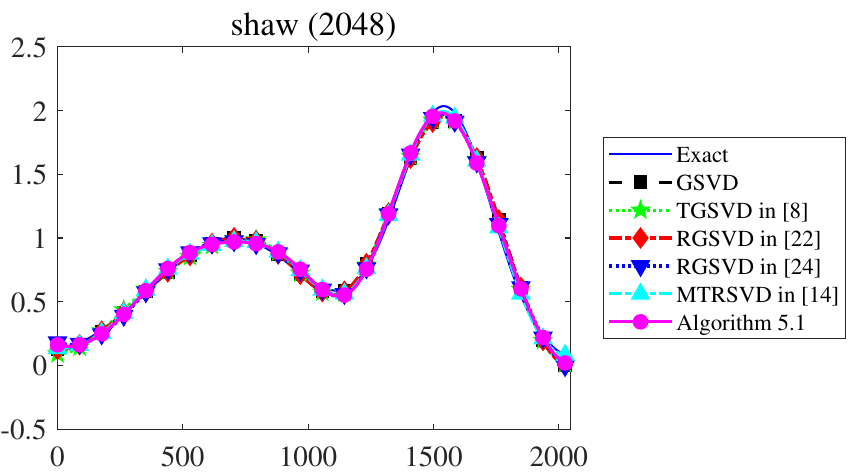}
	}
	\caption{Algorithm \ref{alg:alg4}  (L-curve) for  Example \ref{ex1} with $m=n=2048$}
	\label{Fig:image_over_lcurve}	
\end{figure}

In Tables \ref{Table:image_over_gcv}--\ref{Table:image_over_lcurve} and Figures \ref{Fig:image_over_gcv}--\ref{Fig:image_over_lcurve}, we report the numerical results for Example \ref{ex1}. We can observe  from the first six lines in Tables \ref{Table:image_over_gcv}--\ref{Table:image_over_lcurve} and Figures \ref{Fig:image_over_gcv}--\ref{Fig:image_over_lcurve} that all algorithms produce acceptable approximate solutions with the GCV and L-curve methods. Since the  GCV and L-curve methods are comparable, we use the GCV method for selecting the parameter $\lambda$ in the later numerical tests.
 We also see that Algorithm \ref{alg:alg4} demonstrates  much efficiency over the other randomized algorithms. When the numerical rank of $A$ is large (e.g., {\tt heat} and {\tt phillips}), the  other randomized algorithms still need to solve a large-scale    problem while Algorithm \ref{alg:alg4} only needs to solve a small-scale problem. One may reduce the number of samples by increase  
$\epsilon$ and the solution accuracy may be affected. How to choose an optimal $\epsilon$ needs  further study.

\begin{example} \label{ex2}
{\rm (underdetermined case from Hansen Tools \cite{Hansen})}
We consider the test examples  {\tt baart}, {\tt deriv2}, {\tt foxgood}, {\tt gravity}, {\tt heat} and {\tt shaw} and in Example \ref{ex1} and the underdetermined example {\tt parallax}. Also,  the vector $\bb$ is constructed as in Example \ref{ex1}.
\end{example}

\begin{table}[!ht]
 \centering 
 \begin{threeparttable}{ \scriptsize
\caption{The comparison results (GCV) for Example \ref{ex2} with $n=2048$}
\label{Table:image_under_gcv}
\begin{tabular}{l|ccccccc} 
  \hline
  $n=2048$ & {\tt baart}  & {\tt deriv2} & {\tt foxgood} & {\tt gravity} & {\tt heat}  & {\tt parallax}  & {\tt shaw}   \\
  \hline
$E_{\rm gsvd}$  &	1.02E-01	&	1.70E-02	&	2.91E-02	&	1.41E-02	&	2.95E-02	&	--	&	4.48E-02	\\
$E_{\rm tgsvd}$ in \cite{tikhonov} &	1.03E-01	&	8.93E-02	&	5.56E-02	&	1.78E-02	&	4.06E-02	&	--	&	5.31E-02	\\
$E_{\rm rgsvd}$ in \cite{2016}  &	1.15E-01	&	3.46E-02	&	1.24E-02	&	1.09E-02	&	8.79E-02	&	--	&	4.98E-02	\\
$E_{\rm rgsvd}$ in \cite{xiang2015}  &	1.10E-01	&	3.32E-02	&	1.11E-02	&	1.41E-02	&	7.90E-02	&	--	&	4.54E-02	\\
$E_{\rm mtrsvd}$ in \cite{jia2018}  &	1.03E-01	&	8.52E-02	&	5.58E-02	&	1.69E-02	&	4.00E-02	&	-- &	4.99E-02	\\
$E_{\rm alg. \;\ref{alg:alg4} }$ &	1.08E-01	&	4.35E-02	&	2.11E-02	&	6.27E-03	&	6.33E-02	&	--	&	4.93E-02	\\
$t_{\rm gsvd}$  &	2.8439 	&	3.0824 	&	2.8311 	&	2.8037 	&	2.7647 	&	2.1806 	&	2.8295 	\\
$t_{\rm tgsvd}$ in \cite{tikhonov} &	2.7643 	&	3.3479 	&	2.8672 	&	2.6540 	&	2.8217 	&	2.1437 	&	2.7086 	\\
$t_{\rm rgsvd}$ in \cite{2016}  &	0.2665 	&	0.3987 	&	0.2452 	&	0.2822 	&	0.4396 	&	0.0688 	&	0.2503 	\\
$t_{\rm rgsvd}$ in \cite{xiang2015}  &	0.2674 	&	0.3031 	&	0.2233 	&	0.3463 	&	0.5020 	&	0.0707 	&	0.2714 	\\
$t_{\rm mtrsvd}$ in \cite{jia2018}  &	0.3334 	&	2.0809 	&	2.0612 	&	0.8413 	&	0.7856 	&	1.0518 	&	1.0861 	\\
$t_{\rm alg.\; \ref{alg:alg4} }$ &	0.0719 	&	0.0772 	&	0.0852 	&	0.0590 	&	0.0683 	&	0.0403 	&	0.0556 	\\
$l_{s}$&	4	&	5	&	3	&	13	&	24	&	9	&	8	\\
$l_{1}\times l_{2}$&	$4\times4$	&	$5\times5$	&	$3\times3$	&	$12\times12$	&	$20\times20$	&	$8\times8$	&	$8\times8$	\\
$\lambda_{\rm gsvd}$  &	3.29E-01	&	1.54E-02	&	1.13E-01	&	7.24E-01	&	7.93E-03	&	9.69E-01	&	1.33E-01	\\
$\lambda_{\rm tgsvd}$ in \cite{tikhonov} &	1.26E-02	&	2.23E-01	&	1.96E-01	&	3.35E-01	&	4.05E-02	&	3.11E-01	&	1.84E-02	\\
$\lambda_{\rm rgsvd}$ in \cite{2016}  &	1.01E+00	&	2.30E-01	&	1.84E+00	&	1.92E-01	&	1.23E-02	&	4.41E-01	&	3.53E-01	\\
$\lambda_{\rm rgsvd}$ in \cite{xiang2015}  &	1.04E+00	&	1.32E-01	&	1.74E+00	&	1.97E-01	&	3.09E-02	&	1.95E-01	&	3.65E-01	\\
$\lambda_{\rm alg. \;\ref{alg:alg4} }$ &	1.04E+00	&	2.65E-01	&	1.92E+00	&	4.60E-01	&	7.13E-02	&	1.42E+00	&	3.62E-01	\\
    \hline
\end{tabular}}
\end{threeparttable}
\end{table}

\begin{figure}[!htbp]
	\centering
	\subfigure[{\tt baart}]
    {\includegraphics[width=0.4\linewidth,height=0.14\textheight]{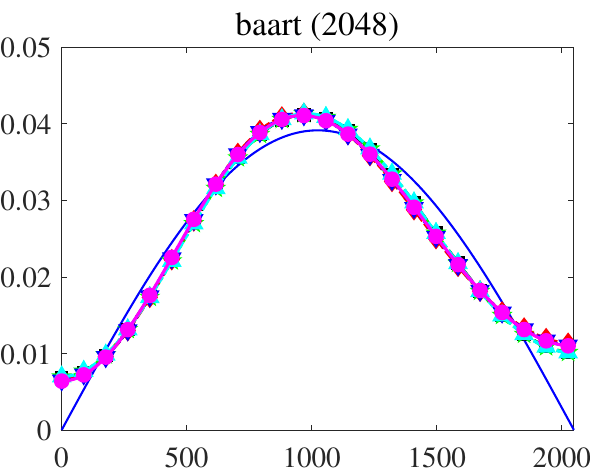}
	}
    \subfigure[{\tt deriv2}]
    {\includegraphics[width=0.4\linewidth,height=0.14\textheight]{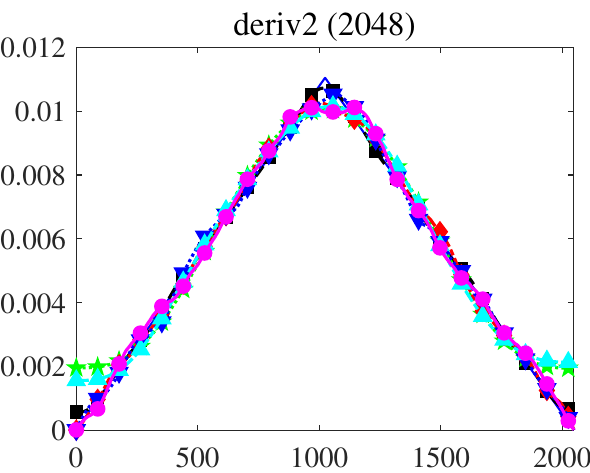}
	}
	\subfigure[{\tt foxgood}]
    {\includegraphics[width=0.4\linewidth,height=0.14\textheight]{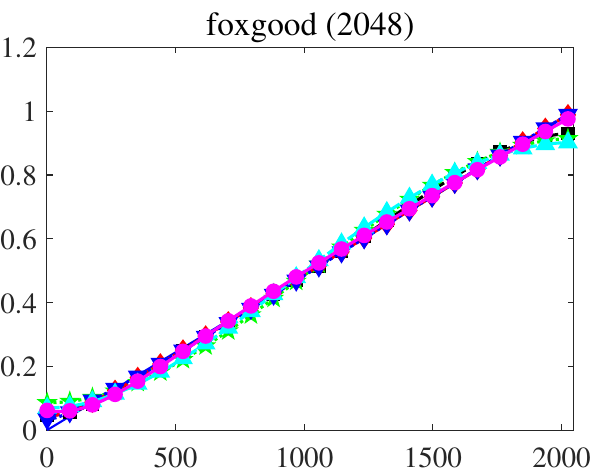}
	}
    \subfigure[{\tt gravity}]
    {\includegraphics[width=0.4\linewidth,height=0.14\textheight]{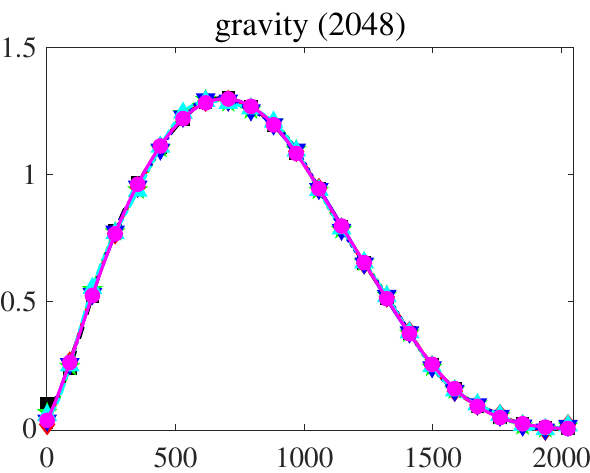}
	}
	\subfigure[{\tt heat}]
    {\includegraphics[width=0.4\linewidth,height=0.14\textheight]{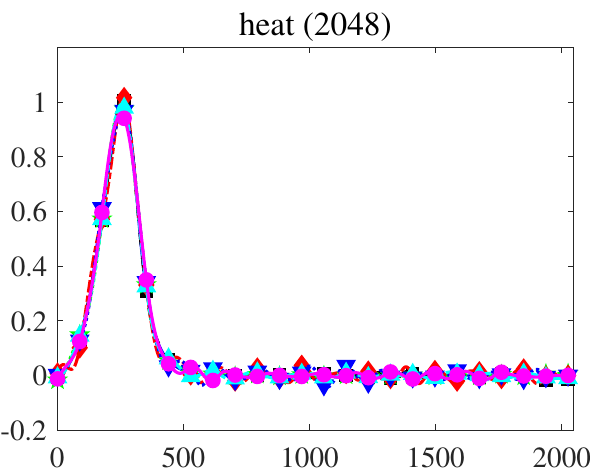}
	}	
	\subfigure[{\tt parallax}]
    {\includegraphics[width=0.4\linewidth,height=0.14\textheight]{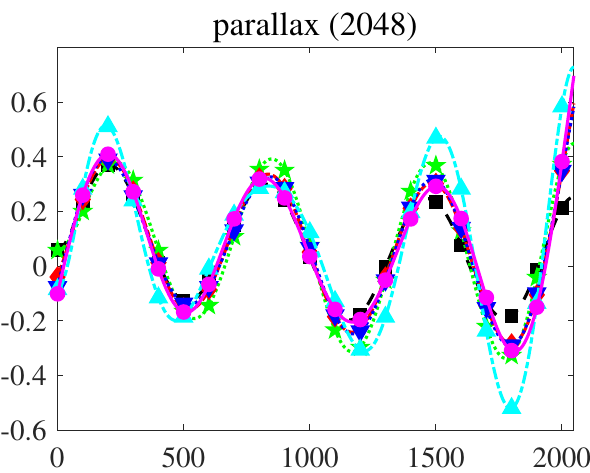}
	}
    \subfigure[{\tt shaw}]
    {\includegraphics[width=0.6\linewidth,height=0.14\textheight]{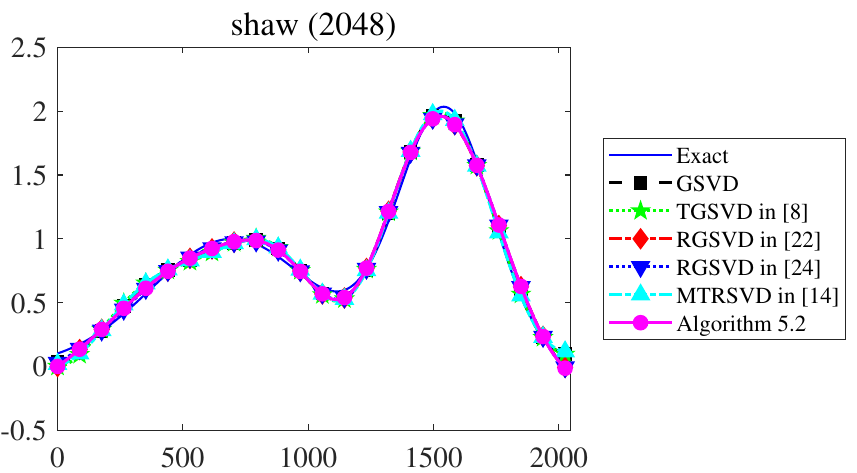}
	}
    \caption{Algorithm \ref{alg:alg5} (GCV) for  Example \ref{ex2} with $n=2048$}
	\label{Fig:image_under_gcv}	
\end{figure}

The numerical results of Example \ref{ex2} are reported in Table \ref{Table:image_under_gcv} and Figure \ref{Fig:image_under_gcv}. Here, the marker ``--" in Table \ref{Table:image_under_gcv}  means that  the exact solution is not available for {\tt parallax}. We see from   Table \ref{Table:image_under_gcv} and Figure \ref{Fig:image_under_gcv} that Algorithm \ref{alg:alg5} is  comparable to the other randomized algorithms in terms of the relative error while Algorithm \ref{alg:alg5} works much efficiently over the randomized algorithms. 

\begin{example} \label{ex3}
This function creates a two-dimensional  phantom test problem from AIR tools \cite{AIR}\footnote{\url{https://github.com/jakobsj/AIRToolsII}}. The size is $N\times N (N=50)$ and the pixel values are between $0$ and $1$. We generate $\bx$ by $\bx ={\tt im(:)}$, where ${\tt im = phantomgallery('ppower', N)}$. Here,  `{\tt ppower}'  means a random image with patterns of nonzero pixels. Then the matrix $A$ is generated by  $A= {\tt paralleltomo(N,theta,p)}$ with specified angles {\tt theta = 0:12:179} and {\tt p = 4*N} parallel rays. Let $\bb=A\bx$. We note that the dimension of $A$ is ${\tt length(theta) \cdot p \times N^2}=3000\times 2500$ and thus   $A\bx=\bb$ is overdetermined. 

\end{example}

\begin{table}[htbp]
\centering {\scriptsize
	\caption{Algorithm \ref{alg:alg4} for Example \ref{ex3}:  {\tt paralleltom} $(3000 \times 2500)$}
\label{example3} 
\begin{tabular}{ccccc}
  \hline
   Method &   $\lambda$ & Relative error &Running time & $l_s/(l_1\times l_2)$\\
   \hline
  GSVD                       & 5.17E-11   & 0.2567 & 6.75 s & --\\
  TGSVD in \cite{tikhonov}   & 5.17E-11   & 0.2568 & 9.85 s & 940\\
  RGSVD in \cite{2016}       & 7.44E-03   & 0.2571 & 5.92 s & 940\\
  RGSVD in \cite{xiang2015}  & 2.58E-03   & 0.2340 & 2.27 s & 940\\
  MTRSVD in \cite{jia2018}   & --       & 0.2569 & 4.92 s & 940\\
  Algorithm \ref{alg:alg4}   & 2.53E-03   & 0.2339 & 0.88 s & $941\times 941$\\
         \hline
\end{tabular}}
\end{table}

\begin{figure}[htbp]
	\centering
	\subfigure[Exact image]{
		\includegraphics[width=4cm,height=0.14\textheight]{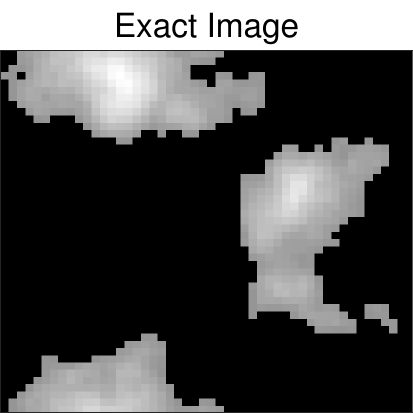}
	}	
	\subfigure[GSVD]{
		\includegraphics[width=4cm,height=0.14\textheight]{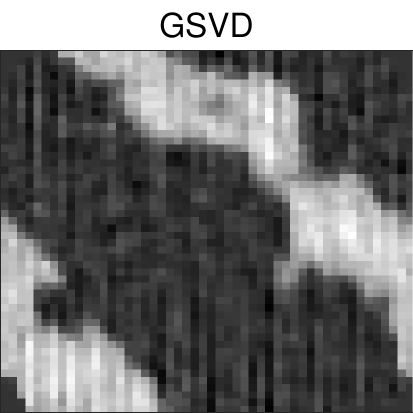}
	}	
	\subfigure[TGSVD in \cite{tikhonov}]{
		\includegraphics[width=4cm,height=0.14\textheight]{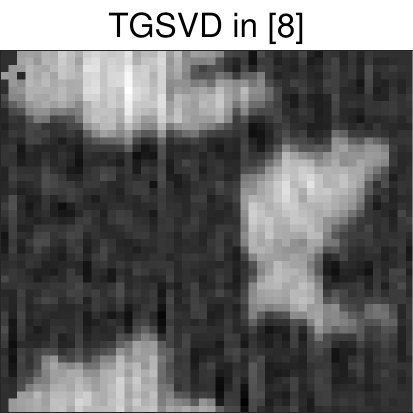}
	}	
	\subfigure[RGSVD in \cite{2016}]{
		\includegraphics[width=4cm,height=0.14\textheight]{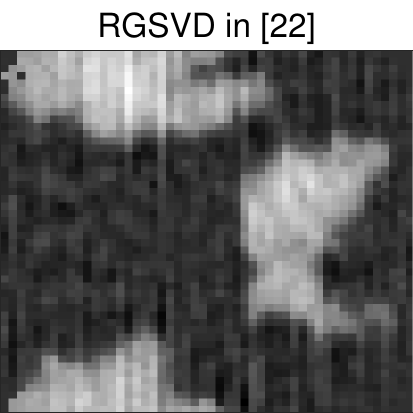}
	}
	\subfigure[RGSVD in \cite{xiang2015}]{
		\includegraphics[width=4cm,height=0.14\textheight]{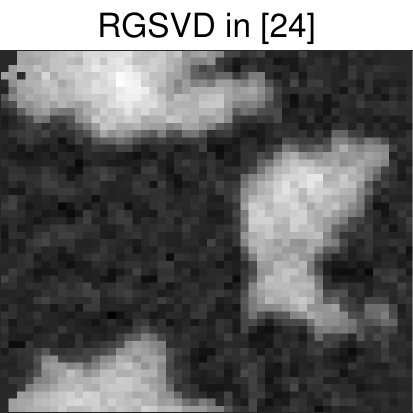}
	}
	\subfigure[MTRSVD in \cite{jia2018}]{
		\includegraphics[width=4cm,height=0.14\textheight]{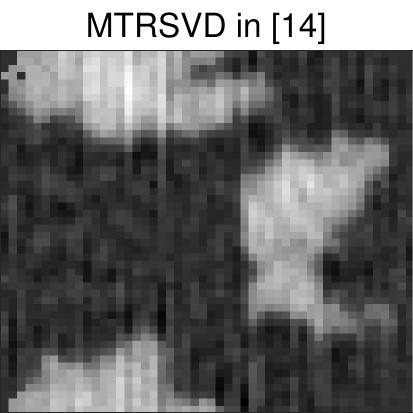}
	}
    \subfigure[Alg. \ref{alg:alg4}]{
		\includegraphics[width=4cm,height=0.14\textheight]{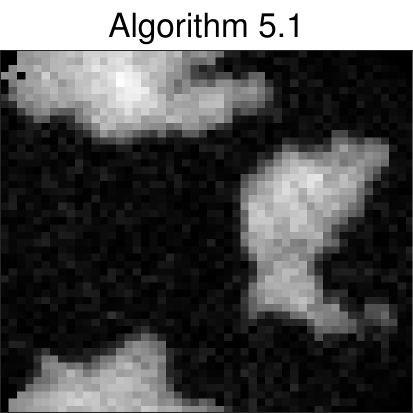}
	}
\caption{Images of Example \ref{ex3}:  {\tt ppower} $(3000 \times 2500)$}\label{fig3}
\end{figure}

The numerical results of Example \ref{ex3} are shown in Table \ref{example3} and the true and approximate solutions of Example \ref{ex3} are displayed in Figure \ref{fig3}. 
Here, the marker ``--" in Table  \ref{example3}  means that  the corresponding value   is not available. We see from  Table \ref{example3} that  Algorithm \ref{alg:alg4}  is much efficient over the other randomized algorithms with comparable relative errors. We also observe from Figure \ref{fig3} that it seems that the approximate solution computed by  Algorithm \ref{alg:alg4} is more robust than the other algorithms.

\section{Concluding remarks}\label{sec7}
In this paper, a two-sided uniformly randomized GSVD algorithm  is proposed for solving the large-scale discrete ill-posed problem with the general Tikhonov regularization. By using two-sided uniform random sampling,  the proposed randomized GSVD algorithm can reduce the scale of the ill-posed problem and thus improve the efficiency in terms of computing time and memory storage. Numerical experiments show the  efficiency of the proposed algorithm.

\end{document}